\documentclass[12pt,a4paper]{article}
\usepackage[utf8]{inputenc}
\usepackage{amsmath, amssymb, amsthm, verbatim, hyperref}
\usepackage{mathrsfs}
\usepackage[dvipsnames]{xcolor}
\usepackage{ragged2e}
\usepackage{marginnote}
\usepackage{enumerate}
\usepackage{makecell}
\usepackage{longtable}
\usepackage{epsfig}
\usepackage{tikz}
\usepackage{ytableau}
\usepackage{hyperref}
\usepackage{caption}
\usepackage{subcaption}
\hypersetup{hidelinks}
\usepackage{fancyvrb}
\usepackage{tikz-cd}
\usepackage[left=2.50cm, right=2.50cm, top=2.00cm, bottom=2.00cm]{geometry}
\usetikzlibrary {arrows.meta}
\usepackage{ifthen}

\newcommand{\cross}[3]{
    \draw[thick] (-0.5+#1, -#2) to (0.5+#1, -#2+1);
    \draw[thick] (-0.5+#1, -#2+1) to (0.5+#1, -#2);
    \foreach \j in {0,...,#3}{
    \ifthenelse{\equal{\j}{\the\numexpr #2-1} \OR \j = #2}{}{
    \draw[thick] (-0.5+#1,-\j ) to (0.5+#1,-\j);}
    }
}

\newcommand{\uncross}[3]{
    \draw[thick] [bend right = 90, looseness=1.25] (-0.5+#1, -#2) to (-0.5+#1, -#2+1);
    \draw[thick] [bend right = 90, looseness=1.25] (0.5+#1, -#2+1) to (0.5+#1, -#2);
    \foreach \j in {0,...,#3}{
    \ifthenelse{\equal{\j}{\the\numexpr #2-1} \OR \j = #2}{}{
    \draw[thick] (-0.5+#1,-\j ) to (0.5+#1,-\j);}
    }
}

\def\nodd{[n]_{\text{odd}}}
\def\Cov{\textsf{Cov}}
\def\Imm{\operatorname{Imm}}
\def\TL{\operatorname{TL}}

\def\startT{\textsf{start}_T}
\def\endT{\textsf{end}_T}

\definecolor{goodgreen}{rgb}{0.01, 0.75, 0.24}

\theoremstyle{plain}

\newtheorem{thm}{Theorem}[section]
\newtheorem{prop}[thm]{Proposition}
\newtheorem{cor}[thm]{Corollary}
\newtheorem{lemma}[thm]{Lemma}

\theoremstyle{definition}
\newtheorem{definition}[thm]{Definition}
\newtheorem{example}[thm]{Example}
\theoremstyle{remark}

\newtheorem{remark}[thm]{Remark}

\newcommand{\ZZ}{\mathbb{Z}}

\newcommand{\CC}{\mathbb{C}}

\newcommand{\eps}{\varepsilon}

\title{\vspace{-1em} 
Temperley–Lieb Crystals}
\author{\vspace{-2em} \\
Son Nguyen, Pavlo Pylyavskyy}
\date{\vspace{-3em}}

\allowdisplaybreaks

\begin{document}
\ytableausetup{centertableaux}

\maketitle

    \begin{abstract}
    Elements of Lusztig's dual canonical bases are Schur-positive when evaluated on (generalized) Jacobi-Trudi matrices. This deep property was proved by Rhoades and Skandera, relying on a result of Haiman, and ultimately on the (proof of) Kazhdan-Lusztig conjecture. For a particularly tractable part of the dual canonical basis - called Temperley-Lieb immanants - we give a generalization of Littlewood-Richardson rule: we provide a combinatorial interpretation for the coefficient of a particular Schur function in the evaluation of a particular Temperley-Lieb immanant on a particular Jacobi-Trudi matrix. For this we introduce shuffle tableaux, and apply Stembridge's axioms to show that certain graphs on shuffle tableaux are type $A$ Kashiwara crystals.  
    \end{abstract}

    \begin{figure}[h!]
        \centering
        \includegraphics[scale = 0.65]{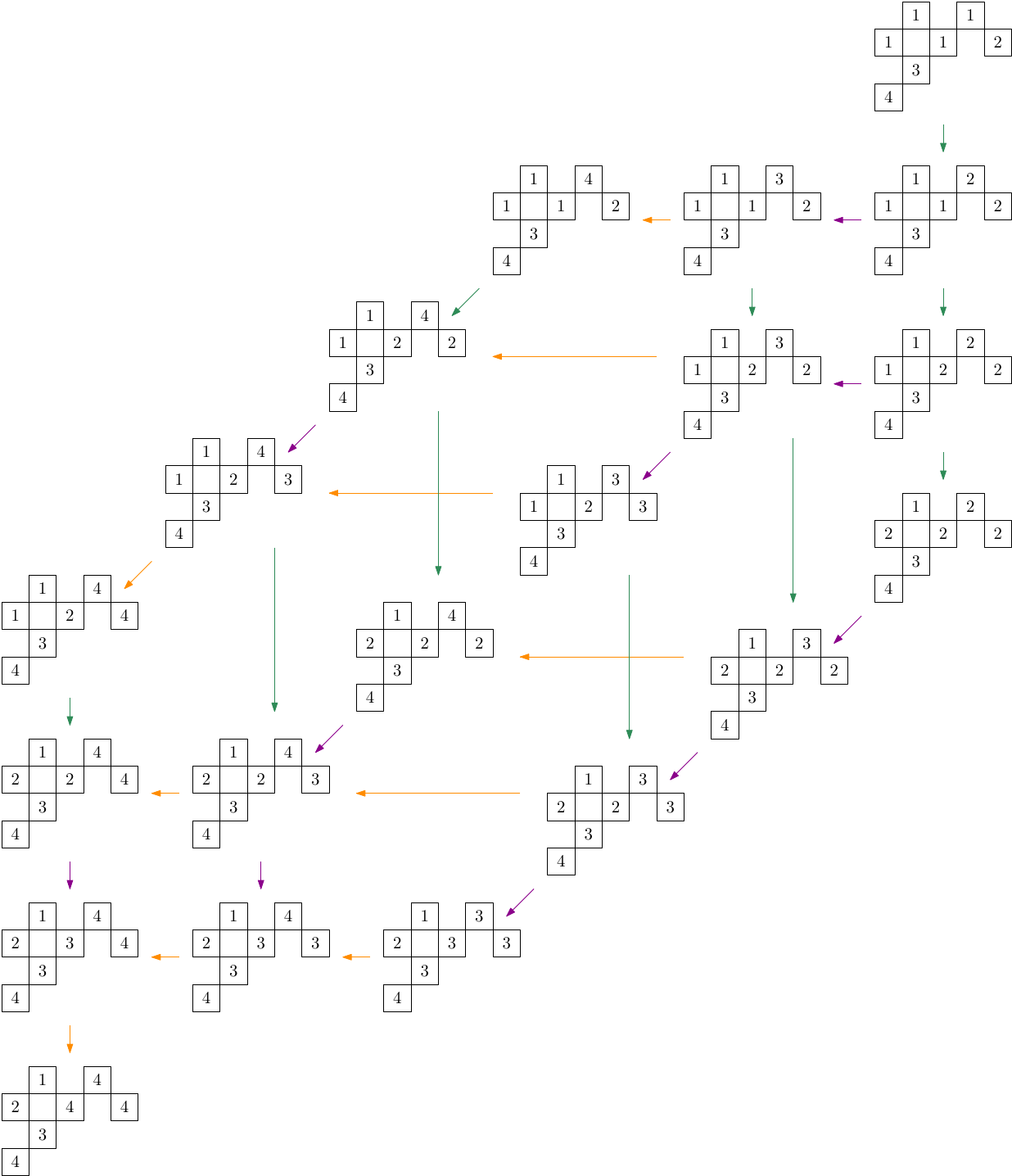}
    \end{figure}

\newpage

\tableofcontents


\section{Introduction}\label{sec:intro}

In his groundbreaking 1990 paper \cite{lusztig1990canonical} Lusztig has introduced {\it {canonical bases}} of quantum groups. One can study canonical bases in the dual setting of quantized coordinate rings, where one can also specialize $q=1$ to obtain the {\it {dual canonical bases}} of coordinate rings of algebraic groups. Study of canonical and dual canonical bases inspired several major developments in mathematics, such as Kashiwara crystal bases \cite{kashiwara1991crystal, kashiwara1993global} and Fomin-Zelevinsky cluster algebras \cite{fomin2002cluster, berenstein1996parametrizations}. 

In type $A$, elements of dual canonical bases are certain polynomials in the matrix algebra ${\mathbb C}[x_{ij}]$, $1 \leq i,j \leq n$. For example, the determinant of the whole matrix and of any of its submatrices is an element of dual canonical basis. Explicit formulas for dual canonical bases in type $A$ were given by Du \cite{du1992canonical, du1995canonical} and by Rhoades-Skandera \cite{rhoades2006kazhdan, skandera2008dual}, where the coefficients are given as evaluations of Kazhdan-Lusztig polynomials. 

Soon after introducing canonical bases, Lusztig has observed \cite{lusztig1994total} that they have a deep and perhaps surprising relation to an old notion of total positivity. A matrix is  totally positive if all of its minors have non-negative determinants. We refer the reader to the survey \cite{fomin1999total} for the history of total positivity. One way to describe the connection is to say that elements of dual canonical bases evaluate to positive numbers on totally positive matrices. This allowed Lusztig to extend the notion of total positivity to any algebraic group and associated flag varieties \cite{lusztig1994total, lusztig1997total}.

A less known, but no less deep positivity property of dual canonical bases in type $A$ was noticed and proved by Haiman \cite{haiman1993hecke} and Rhoades-Skandera \cite{rhoades2006kazhdan}. One of the fundamental results in the theory of symmetric functions is Jacobi-Trudi formula for Schur functions, representing them as determinants of matrices filled with complete homogenous symmetric functions, see \cite{macdonald1998symmetric, stanley2023enumerative} for a standard exposition. Any determinant of a (generalized) Jacobi-Trudi matrix is a skew Schur function, and as such is positive in the basis of Schur functions, thanks to the celebrated Littlewood-Richardson rule \cite{littlewood1934group}. As mentioned above, determinants are elements of the dual canonical bases. It turns out that {\it {all}} of their elements evaluate to Schur positive results on generalized Jacobi-Trudi matrices. In this form this statement was formulated and proved by Rhoades and Skandera in \cite[Proposition 3]{rhoades2006kazhdan}, relying on Haiman's \cite[Theorem 1.5]{haiman1993hecke}. The proof of the latter in turn relies on the (proofs of) Kazhdan-Lusztig conjecture by Beilinson-Bernstein \cite{BB} and Brylinski-Kashiwara \cite{brylinski1981kazhdan}.

Since the available descriptions of elements of dual canonical bases are in terms of evaluations of Kazhdan-Lusztig polynomials, it would appear that one cannot hope for a direct combinatorial proof of the above Schur positivity. It turns out however that there exists a part of dual canonical bases that admits a relatively elementary description. This part was introduced and studied by Rhoades and Skandera in \cite{rhoades2005temperley}, and is called {\it {Temperley-Lieb immanants}}. The reason for such an elementary description is a result of Fan and Green \cite{fan1997monomials} saying that the quotient map from the Hecke algebra to the Temperley-Lieb algebra acts in a particularly nice way on Kazhdan-Lusztig bases elements $C'_w$. This is not true for higher analogs of Temperley-Lieb algebra, see \cite{khovanov1999web}.

Given the existence of such an `elementary' part of the dual canonical bases, it is natural to ask if an elementary proof of Schur positivity exists in this case. In other words, can we prove without relying on the (proof of) Kazhdan-Lusztig conjecture that Temperley-Lieb immanants give Schur-positive results when evaluated on generalized Jacobi-Trudi matrices? This is also interesting because this statement has been used by Lam, Postnikov, and Pylyavskyy \cite{lam2007schur} to prove several deep Schur-positivity conjectures \cite{okounkov1997log, fomin2005eigenvalues, lascoux1997ribbon}, which seem to resist other approaches. For another applications see \cite{dobrovolska2007products}.

In this paper we give such elementary proof. In fact, we accomplish something  stronger: in Theorem \ref{thm:LR-rule} we provide a combinatorial interpretation for the coefficient of a particular Schur function in the evaluation of a particular Temperley-Lieb immanant on a particular generalized Jacobi-Trudi matrix. As explained in Remark \ref{remark:usualLR}, since the determinant is one of the Temperley-Lieb immanants, this gives a generalization of the Littlewood-Richardson rule. 

The main tool we use is a powerful technique introduced by Stembridge \cite{stembridge2003local}: by verifying certain local properties, it is possible to prove that a given graph is a type $A$ Kashiwara crystal. Since Schur functions are generating functions over vertices of such crystals, this allows one to prove Schur positivity results, if one can find appropriate objects and crystal operators on them. In our case the objects are called {\it {shuffle tableaux}}, and the crystal operators are defined using a variation of the usual bracketing procedure. The hardest and most technical part of the proof is to verify that Stembridge's axioms hold for the resulting graphs. 

The paper proceeds as follows. In Section \ref{sec:def} we review relevant definitions of Jacobi-Trudi matrices, Temperley-Lieb immanants, and the way both relate to planar networks. In Section \ref{sec:shuffle-tableau} we define shuffle tableaux, and associate to each of them its Temperley-Lieb type. In Section \ref{sec:crys} we define crystal operators on shuffle tableaux, thus obtaining Temperley-Lieb crystals. We prove that crystal operators do not change the Temperley-Lieb type of a shuffle tableaux. In Section \ref{sec:axiom} we prove that Stembridge's axioms hold for Temperley-Lieb crystals. Finally in Section \ref{sec:LR-rule} we derive a generalized Littlewood-Richardson rule for Temperley-Lieb immanants.

\section{Definitions}\label{sec:def}

\subsection{Generalized Jacobi-Trudi matrix and wiring}\label{subsec:JTmatrix}

    A \textit{planar network of order $n$} is an acyclic planar directed graph in $2n$ (not necessarily distinct) boundary vertices: $n$ sources $L_1,\ldots,L_n$ and $n$ sinks $R_1,\ldots,R_n$. Each edge $e$ has a weight $\omega(e)$, and we define the weight of a (multi)set of edges to be the product of the weight of the edges. Given a planar network $G$ of order $n$, its path matrix is the matrix $A = (a_{ij})$, where
    \[ a_{ij} = \sum_{P}\omega(P), \]
    summing over all paths from source $L_i$ to sink $R_j$.

    A lattice is a (directed) graph whose vertices are points $(i,j) \in \ZZ^2$ and edges are $(i+1,j)\rightarrow (i,j)$ and $(i,j+1)\rightarrow (i,j)$. In other words, the edges are oriented in the West and South directions. Each vertical edge $(i,j+1) \rightarrow (i,j)$ has weight $1$ while each horizontal edge $(i+1,j) \rightarrow (i,j)$ has weight $x_i$. The weight of a lattice path $\pi$, denoted $\omega(\pi)$, is the product of the weight of its edges. It is well-known that the sum of the weight of all lattice paths from $(\mu_i,\infty)$ to $(\nu_j,1)$ is the complete homogeneous symmetric polynomial $h_{\mu_i-\nu_j}$.

    Given partitions $\mu = (\mu_1,\ldots,\mu_n)$ and $\nu = (\nu_1,\ldots,\nu_n)$, the \textit{generalized Jacobi-Trudi matrix} $A_{\mu,\nu}$ is the finite matrix whose $i,j$ entry is $h_{\mu_i - \nu_j}$. By convention, we have $h_0 = 1$ and $h_m = 0$ for all $m < 0$.

    \begin{example}
        When $\mu = (9,6,6,4,3)$ and $\nu = (7,5,4,1,1)$, the generalized Jacobi-Trudi matrix is the following
        \[ \left(\begin{matrix}
            h_2 & h_4 & h_5 & h_8 & h_8 \\
            0 & h_1 & h_2 & h_5 & h_5 \\
            0 & h_1 & h_2 & h_5 & h_5 \\
            0 & 0 & 1 & h_3 & h_3 \\
            0 & 0 & 0 & h_2 & h_2 \\
        \end{matrix}\right) .\]
        One can check that $A_{\mu,\nu}$ is the path matrix for the network $G_{\mu,\nu}$ on lattice when the sources are $(\mu_1,\infty),\ldots,(\mu_n,\infty)$ and the sinks are $(\nu_1,1),\ldots,(\nu_n,1)$.
    \end{example}

    A \textit{(generalized) wiring} $H$ on a network $G$ is a family of $n$ paths $\pi_1,\ldots,\pi_n$ such that
    \begin{itemize}
        \item $\pi_i$ goes from $L_i$ to $R_i$ for all $i$, and
        \item no three paths share a vertex.
    \end{itemize}
    Figure \ref{fig:wiring} gives an example of a wiring of $G_{\mu,\nu}$ when $\mu = (9,6,6,4,3)$ and $\nu = (7,5,4,1,1)$. Note that the sources $L_2$ and $L_3$ are at the same vertex $(6,\infty)$, and the sinks $R_4$ and $R_5$ are at the same vertex $(1,1)$. By definition, wirings may only exist when no three sources nor three sinks coincide. Each edge in the lattice belongs to at most two paths. If an edge belongs to two paths, we say that edge is \textit{doubly covered}. The weight of a wiring is the product of the weight of all paths, i.e.
    \[ \omega(H) = \prod_{i = 0}^n\omega(\pi_i). \]

    \begin{figure}[h!]
        \centering
        \includegraphics[scale = 0.8]{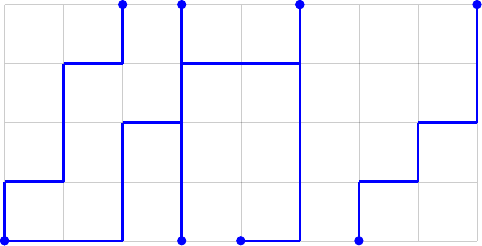}
        \caption{Generalized wiring}
        \label{fig:wiring}
    \end{figure}

    \begin{remark}
        It is common in the literature to orient the edges in the lattice in the North and East directions. In this paper, however, we reverse the directions to be consistent with the definition of generalized Jacobi-Trudi matrices.
    \end{remark}

    \begin{remark}
        The reason why we require that no three paths share a vertex is that when such three paths exist, Rhoades-Skandera showed in \cite{rhoades2005temperley} that the immanant, defined in Section \ref{subsec:TLimmanant}, is zero. Therefore, we do not need to consider that case. In particular, we will only consider generalized Jacobi-Trudi matrices from partitions $\mu$ and $\nu$ with no three equal parts.
    \end{remark}

\subsection{Temperley-Lieb immanant}\label{subsec:TLimmanant}

    The \textit{Hecke algebra} $H_n(q)$ is the $\CC$-algebra associated with the symmetric group $S_n$. $H_n(q)$ has the standard basis $\{ T_w~|~w\in S_n \}$, where $T_w = T_{s_{i_1}}\ldots T_{s_{i_k}}$ for a reduced decomposition $w = s_{i_1}\ldots s_{i_k}$. The generators of $H_n(q)$ have the following relations
    \begin{align*}
        T_{s_i}^2 &= (q-1)T_{s_i} + q, &\hspace{-5em}&\text{for $i = 1,\ldots,n-1$}, \\
        T_{s_i}T_{s_j}T_{s_i} &= T_{s_j}T_{s_i}T_{s_j}, &\hspace{-5em}&\text{if $|i-j| = 1$}, \\
        T_{s_i}T_{s_j} &= T_{s_j}T_{s_i},&\hspace{-5em}&\text{if $|i-j| > 1$}.
    \end{align*}
    The \textit{Temperley-Lieb algebra} $TL_n(\xi)$ is the $\CC[\xi]$-algebra generated by $t_1,\ldots,t_{n-1}$ subject to the relations
    \begin{align*}
        t_i^2 &= \xi t_i,&\hspace{-5em}&\text{for $i = 1,\ldots,n-1$}, \\
        t_it_jt_i &= t_i,&\hspace{-5em}&\text{if $|i-j| = 1$}, \\
        t_it_j &= t_jt_i,&\hspace{-5em}&\text{if $|i-j| > 1$}.
    \end{align*}
    A \textit{$321$-avoiding permutation} is a permutation in $S_n$ that has no reduced decomposition of the form $\cdots s_is_js_i\cdots$ where $|i-j| = 1$. Because of the second relation, the Temperley-Lieb algebra has a natural basis $\{t_w~|~\text{$w$ is a 321-avoiding permutation}\}$, where $t_w := t_{i_1}\ldots t_{i_k}$ for a reduced decomposition $w = s_{i_1}\ldots s_{i_k}$. The map
    \[ \theta~:~T_{s_i} \rightarrow t_i - 1 \]
    determines a homomorphism $\theta~:~H_n(1) \rightarrow TL_n(2)$. For any permutation $v\in S_n$ and any 321-avoiding permutation $w\in S_n$, let $f_w(v)$ be the coefficient of $t_w\in TL_n(2)$ in the basis expansion of $\theta(T_v) = (t_{i_1}-1)\ldots (t_{i_k}-1)$ for a reduced decomposition $v = s_{i_1}\ldots s_{i_k}$. In \cite{rhoades2005temperley}, Rhoades and Skandera defined the \textit{Temperley-Lieb immanant} of an $n\times n$ matrix $X = (x_{ij})$ by
    \[ \Imm_w^{\TL}(X) = \sum_{v\in S_n}f_w(v)x_{1,v(1)}\ldots x_{n,v(n)}. \]

    Products of $t_i$'s in $TL_n(\xi)$ can be computed graphically using \textit{Temperley-Lieb diagrams}. The generators $t_1,t_2,t_3,\ldots$ are represented by
    \[ \resizebox{.05\textwidth}{!}{
        \begin{tikzpicture}
            \uncross{1}{1}{3}
            \draw[thick] (1.5,-5) to (0.5, -5);
            \draw (1,-4) node[anchor=center, scale = 1.5] {$\vdots$};
        \end{tikzpicture}
    },\quad 
    \resizebox{.05\textwidth}{!}{
        \begin{tikzpicture}
            \uncross{1}{2}{3}
            \draw[thick] (1.5,-5) to (0.5, -5);
            \draw (1,-4) node[anchor=center, scale = 1.5] {$\vdots$};
        \end{tikzpicture}
    },\quad 
    \resizebox{.05\textwidth}{!}{
        \begin{tikzpicture}
            \uncross{1}{3}{3}
            \draw[thick] (1.5,-5) to (0.5, -5);
            \draw (1,-4) node[anchor=center, scale = 1.5] {$\vdots$};
        \end{tikzpicture}
    },\ldots\]
    and multiplication is represented by concatenation.
    
    \begin{example}
        The Temperley-Lieb diagram for $t_1t_3t_2t_1t_3\in TL_4$ is
        \[ \resizebox{!}{.15\textwidth}{
            \begin{tikzpicture}
                \draw (0,0) node[anchor=center] {$L_1$};
                \draw (0,-1) node[anchor=center] {$L_2$};
                \draw (0,-2) node[anchor=center] {$L_3$};
                \draw (0,-3) node[anchor=center] {$L_4$};
                \draw (6,0) node[anchor=center] {$R_1$};
                \draw (6,-1) node[anchor=center] {$R_2$};
                \draw (6,-2) node[anchor=center] {$R_3$};
                \draw (6,-3) node[anchor=center] {$R_4$};
                \filldraw[black] (0.5,0) circle (2pt);
                \filldraw[black] (0.5,-1) circle (2pt);
                \filldraw[black] (0.5,-2) circle (2pt);
                \filldraw[black] (0.5,-3) circle (2pt);
                \filldraw[black] (5.5,0) circle (2pt);
                \filldraw[black] (5.5,-1) circle (2pt);
                \filldraw[black] (5.5,-2) circle (2pt);
                \filldraw[black] (5.5,-3) circle (2pt);
                \uncross{1}{1}{3}
                \uncross{2}{3}{3}
                \uncross{3}{2}{3}
                \uncross{4}{1}{3}
                \uncross{5}{3}{3}
            \end{tikzpicture}
        } \]
    \end{example}

    One can observe that each diagram consists of a noncrossing matching of the boundary vertices and internal loops. If two Temperley-Lieb diagrams have the same matching and the same number of internal loops, then the corresponding products are equal to each other. If the diagram of $a$ is obtained from the diagram of $b$ by removing $k$ internal loops, then $b = \xi^ka$ in $TL_n$. Thus, we have a natural bijection between basis elements of $TL_n$ and noncrossing matchings of $2n$ vertices.

    Given a generalized wiring, one can also obtain a Temperley-Lieb diagram as follows.
    \begin{enumerate}
        \item Contract any doubly covered subpath to a single vertex.
        \item If two sources coincide at the same vertex $v$, create two new sources, each having one edge going to $v$. If two sinks coincide at the same vertex $v$, create two new sinks, each having one edge coming from $v$.
        \item For each vertex $v$ of indegree two and outdegree two, create vertex $v'$ with indegree two and vertex $v''$ with outdegree two.
    \end{enumerate}

    \begin{example}
        One can check that the wiring in Figure \ref{fig:wiring} corresponds to the following diagram.
        \[ \resizebox{!}{0.2\textwidth}{
            \begin{tikzpicture}
                \draw (0,0) node[anchor=center] {$L_1$};
                \draw (0,-1) node[anchor=center] {$L_2$};
                \draw (0,-2) node[anchor=center] {$L_3$};
                \draw (0,-3) node[anchor=center] {$L_4$};
                \draw (0,-4) node[anchor=center] {$L_5$};
                \draw (4,0) node[anchor=center] {$R_1$};
                \draw (4,-1) node[anchor=center] {$R_2$};
                \draw (4,-2) node[anchor=center] {$R_3$};
                \draw (4,-3) node[anchor=center] {$R_4$};
                \draw (4,-4) node[anchor=center] {$R_5$};
                \filldraw[black] (0.5,0) circle (2pt);
                \filldraw[black] (0.5,-1) circle (2pt);
                \filldraw[black] (0.5,-2) circle (2pt);
                \filldraw[black] (0.5,-3) circle (2pt);
                \filldraw[black] (0.5,-4) circle (2pt);
                \filldraw[black] (3.5,0) circle (2pt);
                \filldraw[black] (3.5,-1) circle (2pt);
                \filldraw[black] (3.5,-2) circle (2pt);
                \filldraw[black] (3.5,-3) circle (2pt);
                \filldraw[black] (3.5,-4) circle (2pt);
                \uncross{1}{2}{4}
                \uncross{2}{3}{4}
                \uncross{3}{4}{4}
            \end{tikzpicture}
        } \]
    \end{example}
    Given a wiring $H$, we follow the notations in \cite{rhoades2005temperley} and define $\psi(H)$ to be the noncrossing matching of $H$ and $\epsilon(H)$ to be the number of loops in $H$. We say that $\psi(H)$ is the Temperley-Lieb type of $H$. Rhoades and Skandera proved the following theorem.

    \begin{thm}[{\cite[Theorem 3.1]{rhoades2005temperley}}]\label{thm:imm-formula}
        Let $G$ be a planar network of order $n$ and $A$ be its path matrix. For any basis element $\tau$ of $TL_n(2)$,
        \[ \Imm^{\TL}_\tau(A) = \sum_{H}2^{\epsilon(H)}\omega(H), \]
        where the sum is over all wiring $H$ of $G$ such that $\psi(H) = \tau$.
    \end{thm}



\subsection{Colored paths}\label{subsec:paths}

    For two subsets $I,J\in [n]$ such that $|I| = |J|$, we color the boundary vertices as follows. Color $L_i$ black if $i\in I$ and white otherwise. Color $R_j$ white if $j\in J$ and black otherwise. A Temperley-Lieb basis element is \textit{compatible} with $(I,J)$ if each strand in the diagram has one black endpoint and one white endpoint. 

    \begin{definition}\label{def:colored-cover}
        Given a wiring $H$ and $\tau = \psi(H)$. Let $I, J \subseteq [n]$ with $|I| = |J|$ be a coloring that is compatible with $\tau$. A colored cover of $H$ by $(I,J)$ is a coloring of the edges in $H$ by two colors red and blue that forms two families of red and blue paths satisfying:
        \begin{enumerate}
            \item every edge in $H$ is colored; in particular, the doubly covered edges are colored by two colors;
            \item the red paths are noncrossing paths from $\{L_i~|~i\in I\}$ to $\{R_j~|~j\in J\}$, and the blue paths are noncrossing paths from $\{L_i~|~i\in \Bar{I}\}$ to $\{R_j~|~j\in \Bar{J}\}$.
        \end{enumerate}
        If $x$ is a colored cover of $H$, then we abuse notation and define $\omega(x) = \omega(H)$, $\psi(x) = \psi(x)$, and $\epsilon(x) = \epsilon(H)$. We also label the colored paths $p_1,p_2,\ldots,p_n$ where $p_i$ starts from $L_i$.
    \end{definition}

    \begin{example}
        Figure \ref{fig:colored-cover} shows two colored covers of the same wiring by the coloring $I = \{1,4\}, J = \{1,2\}$.

        \begin{figure}[h!]
            \centering
            \includegraphics[scale = 0.8]{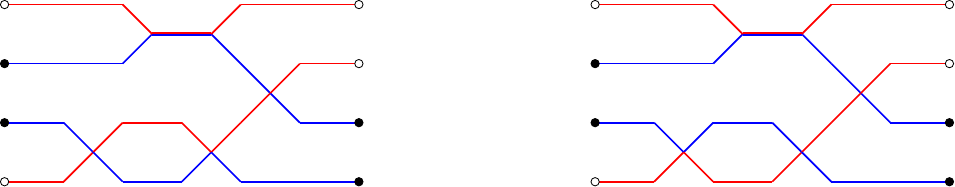}
            \caption{Two colored covers of a wiring}
            \label{fig:colored-cover}
        \end{figure}
    \end{example}

    \begin{prop}[{\cite[Proposition 2.1]{skandera2004inequalities}}]\label{prop:colored_path}
        Given a wiring $H$ and $\tau = \psi(H)$. Let $I, J \subseteq [n]$ with $|I| = |J|$ be a coloring that is compatible with $\tau$. Let $\Cov(H)_{I,J}$ be the set of colored covers of $H$ by $(I,J)$. Then,
        \[ |\Cov(H)_{I,J}| = 2^{\epsilon(H)}. \]
        On the other hand, if $(I,J)$ is not compatible with $\tau$, then $|\Cov(H)_{I,J}| = 0$.
    \end{prop}

    \begin{proof}[Proof sketch]
        We can construct each colored cover as follows. Start from $L_i$ for $i\in I$ (resp. $i\notin I$), color the edges red (resp. blue) when traversing forward and blue (resp. red) when traversing backward. Similarly, start from $R_j$ for $j\in J$ (resp. $j\notin J$), color the edges red (resp. blue) when traversing backward and blue (resp. red) when traversing forward. Color the doubly covered edges both red and blue. Finally, for each internal loop, pick an orientation and color the edges red (resp. blue) when traversing forward (resp. backward). There are two possible orientations for each loop, so there are $2^{\epsilon(H)}$ colored covers in total.
    \end{proof}

    \begin{cor}\label{cor:colored-cover}
        For any basis element $\tau$ of $TL_n(2)$ and any coloring $(I,J)$ that is compatible with $\tau$, we have
        \[ \Imm^{\TL}_\tau(A) = \sum_{\psi(H) = \tau} \sum_{x\in\Cov(H)_{I,J}} \omega(x). \]
    \end{cor}

    When $I = J$, we have a special type of colored cover.

    \begin{definition}\label{def:canonical-cc}
        When $I = J$, every wiring $H$ such that $\tau = \psi(H)$ is compatible with $(I,J)$ has a \textit{canonical colored cover}. This colored cover is obtained by treating each $\resizebox{!}{0.8em}{\begin{tikzpicture}
            \cross{1}{1}{1}
        \end{tikzpicture}}$ as $\resizebox{!}{0.8em}{\begin{tikzpicture}
            \draw[thick] [bend right = 90, looseness=1.25] (-0.5, 0) to (0.5, 0);
            \draw[thick] [bend right = 90, looseness=1.25] (0.5, -1) to (-0.5, -1);
        \end{tikzpicture}}$.
    \end{definition}

    \begin{remark}\label{rem:loop-and-order}
        An important property of the canonical colored cover is that if a vertical line is drawn through the wiring, then it intersects paths $p_1,p_2,\ldots,p_n$ in that order. From the proof of Proposition \ref{prop:colored_path}, every other colored cover is obtained from the canonical colored cover by swapping the orientation of some loops. Thus, if for some colored cover of a wiring $H$ by $I = J$, a line drawn through the wiring intersects path $p_i$ before $p_j$ where $i > j$, then the segments of $p_i$ and $p_j$ that intersect the line have to belong to some loops.
    \end{remark}

\subsection{Products of minors}\label{subsec:prod-minor}

    Denote by $\Theta(I,J)$ the set of Temperley-Lieb basis elements that are compatible with $(I,J)$. Let $\Bar{I} := [n] \backslash I$, and let $\Delta_{I,J}(A)$ denote the minor of $A$ in the row set $I$ and column set $J$. We have the following identity.
    \begin{thm}[{\cite[Proposition 4.3]{rhoades2005temperley}}, cf. \cite{skandera2004inequalities}]\label{thm:prod-minor}
        For two subsets $I,J\in [n]$ such that $|I| = |J|$, we have
        \[ \Delta_{I,J}(A)\cdot\Delta_{\Bar{I},\Bar{J}}(A) = \sum_{\tau\in\Theta(I,J)}\Imm^{\TL}_\tau(A). \]
    \end{thm}
    To define our crystal operators, we will make a specific choice of $I,J$: 
    \[ I = J = \nodd =  \{i \in [n]~|~i~\text{is odd}\}. \]
    The reason for this choice of $I,J$ is the following observation.
    \begin{prop}\label{prop:compatible}
        Let $I = J = \nodd$, every type $\tau$ is compatible with $(I,J)$.
    \end{prop}

    \begin{proof}
        Since $I = J = \nodd$, the colors of the boundary vertices alternate between black and white. For any strand in any perfect matching $\tau$, the number of boundary vertices on each side of the strand is even. Thus, the two endpoints of the strand have different colors.
    \end{proof}
    \begin{cor}\label{cor:all-type}
        For $I = J = \nodd$,
        \[ \Delta_{I,J}(A)\cdot\Delta_{\Bar{I}\Bar{J}}(A) = \sum_{\tau}\Imm^{\TL}_\tau(A), \]
        where the sum is over all Temperley-Lieb basis elements.
    \end{cor}


\section{Shuffle tableaux}\label{sec:shuffle-tableau}

    We now turn our attention to tableaux. We will index the rows of the tableaux from top to bottom, and the columns from left to right. We index by $(i,j)$ the entry on row $i$ column $j$ of the tableau. For a partition $\mu = (\mu_1,\ldots,\mu_\ell)$, we say $\ell(\mu) = \ell$.

    Given sequences $\{f_1 \leq f_2 \leq \ldots \leq f_m\}$ and $\{g_1 \leq g_2 \leq \ldots \leq g_n\}$, we say that $f$ \textit{interlaces} $g$ if $m = n-1$ and
    \[ g_1 \leq f_1 \leq g_2 \leq f_2 \leq \ldots \leq f_{n-1} \leq g_n. \]
    We say that $f$ \textit{alternates left of} $g$ if $m = n$ and
    \[ f_1 \leq g_1 \leq f_2 \leq g_2 \leq \ldots \leq f_n \leq g_n. \]
    Finally, we say $f$ \textit{interleaves} $g$, denoted $f \ll g$, if $f$ either interlaces or alternates left of $g$.
    
    Given two skew shapes $\mu_R/\nu_R$ and $\mu_R/\nu_R$, let $\rho_R = \left(\ell(\mu_R), \ell(\mu_R)-1, \ldots, 1\right)$ and $\rho'_R = \left(\ell(\mu_R), \ell(\mu_R)-1, \ldots, 2\right)$. We say $\mu_B/\nu_B$ interleaves $\mu_R/\nu_R$, denoted $\mu_B/\nu_B \ll \mu_R/\nu_R$, if $\ell(\mu_R) - \ell(\mu_B) = 1$, and
    \[ (\mu_B + \rho'_R) \ll (\mu_R + \rho_R), ~\text{and}~ (\nu_B + \rho'_R) \ll (\nu_R + \rho_{\ell}); \]
    or $\ell(\mu_R) - \ell(\mu_B) = 0$, and
    \[ (\mu_B + \rho_R) \ll (\mu_R + \rho_R), ~\text{and}~ (\nu_B + \rho_R) \ll (\nu_R + \rho_{\ell}). \]

    \begin{example}\label{exp:interleave-shape}
        We have $(7,5,2)/(2,0,0) \ll (7,6,4) / (5,1,0)$ since
        \[ (7,5,2) + (3,2,1) = (10,7,3) \ll (10,8,5) = (7,6,4) + (3,2,1), \]
        and
        \[ (2,0,0) + (3,2,1) = (5,2,1) \ll (8,3,1) = (5,1,0) + (3,2,1). \]
        Similarly, $(7,5)/(2,0) \ll (7,6,4) / (5,1,0)$ since
        \[ (7,5) + (3,2) = (10,7) \ll (10,8,5) = (7,6,4) + (3,2,1), \]
        and
        \[ (2,0) + (3,2) = (5,2) \ll (8,3,1) = (5,1,0) + (3,2,1). \]
    \end{example}

    \begin{definition}[Shuffle tableau]\label{def:shuffle-tableau}
        Given two semi-standard Young tableaux $T_B$ of shape $\mu_B/\nu_B$ and $T_R$ of shape $\mu_R/\nu_R$ such that $\mu_B/\nu_B \ll \mu_R/\nu_R$, the shuffle tableau $T$ of $T_R$ and $T_B$ is constructed as follows.
        \begin{itemize}
            \item Let $a$ be the value of square $(i,j)$ in $T_R$, then fill square $(2i-1,2j-1)$ of $T$ with $a$.
            \item Let $a$ be the value of square    $(i,j)$ in $T_B$, then fill square $(2i,2j)$ of $T$ with $a$.
        \end{itemize}
        In addition, let $\ell = \ell(\mu_R)$, we denote the shape of $T$ by $\mu\oslash\nu$, where
        \[ \mu = ((\mu_R)_1 + \ell, (\mu_B)_1 + \ell, (\mu_R)_2 + \ell - 1, (\mu_B)_2 + \ell - 1, \ldots), \]
        \[ \nu = ((\nu_R)_1 + \ell, (\nu_B)_1 + \ell, (\nu_R)_2 + \ell - 1, (\nu_B)_2 + \ell - 1, \ldots). \]
        Obverse that the number of squares on row $i$ of $T$ is still $\mu_i - \nu_i$. Finally, we denote $\ell(T) := \ell(\mu)$.
    \end{definition}

    The notation $\mu\oslash\nu$ may seem unnecessarily complicated for now, but we will see shortly that there is a bijection between shuffle tableaux of shape $\mu\oslash\nu$ and colored covers of wirings of $G_{\mu,\nu}$.

    \begin{example}\label{exp:shuffle-tableau}
        Let $T_R$ be
        \[ \begin{ytableau} 
        \none & \none & \none & \none & \none & \textcolor{red}{1} & \textcolor{red}{3} \\
        \none & \textcolor{red}{1} & \textcolor{red}{1} & \textcolor{red}{1} & \textcolor{red}{2} & \textcolor{red}{3} \\ 
        \textcolor{red}{1} & \textcolor{red}{2} & \textcolor{red}{2} & \textcolor{red}{3}
        \end{ytableau} \]
        of shape $(7,6,4) / (5,1,0)$ and $T_B$ be
        \[ \begin{ytableau} 
        \none & \none & \textcolor{blue}{2} & \textcolor{blue}{2} & \textcolor{blue}{2} & \textcolor{blue}{4} & \textcolor{blue}{4} \\
        \textcolor{blue}{1} & \textcolor{blue}{2} & \textcolor{blue}{4} & \textcolor{blue}{4} & \textcolor{blue}{4} \\ 
        \textcolor{blue}{3} & \textcolor{blue}{4}
        \end{ytableau} \]
        of shape $(7,5,2) / (2,0,0)$, the shuffle tableau $T$ of $T_R$ and $T_B$ is
        \[ \begin{ytableau} 
        \none & \none & \none & \none & \none & \none & \none & \none & \none & \none & \textcolor{red}{1} & \none & \textcolor{red}{3} \\
        \none & \none & \none & \none & \none & \textcolor{blue}{2} & \none & \textcolor{blue}{2} & \none & \textcolor{blue}{2} & \none & \textcolor{blue}{4} & \none & \textcolor{blue}{4} \\
        \none & \none & \textcolor{red}{1} & \none & \textcolor{red}{1} & \none & \textcolor{red}{1} & \none & \textcolor{red}{2} & \none & \textcolor{red}{3} \\ 
        \none & \textcolor{blue}{1} & \none & \textcolor{blue}{2} & \none & \textcolor{blue}{4} & \none & \textcolor{blue}{4} & \none & \textcolor{blue}{4} \\ 
        \textcolor{red}{1} & \none & \textcolor{red}{2} & \none & \textcolor{red}{2} & \none & \textcolor{red}{3} \\ 
        \none & \textcolor{blue}{3} & \none & \textcolor{blue}{4}
        \end{ytableau} \]
        of shape $\mu\oslash\nu$, where
        \[ \mu = (10,10,8,7,5,3), ~\text{and}~ \nu = (8,5,3,2,1,1). \]
    \end{example}

    \begin{lemma}\label{lem:shuffle-props}
        Given a shuffle tableau $T$ of shape $\mu\oslash\nu$, for each row $i$, let $\startT(i)$ (resp. $\endT(i)$) be the column containing the leftmost (resp. rightmost) square on row $i$.
        \begin{enumerate}
            \item[(a)] The entries in $T$ are weakly increasing in each row and strictly increasing in each column.
            \item[(b)] We have, $\startT(i) \leq \startT(i-1) + 1$ for all $1 < i \leq \ell(T)$, and $\startT(i) \leq \startT(i-2)$ for all $2 < i < \ell(T)$. Similarly, $\endT(i) \leq \endT(i-1) + 1$ for all $1 < i \leq \ell(T)$, and $\endT(i) \leq \endT(i-2)$ for all $2 < i < \ell(T)$.
            \item[(c)] The sequence $(\startT(1),\startT(2),\ldots,\startT(\ell(T)))$ is related to $\nu$ by
            \[ \nu_{2i+1} = \dfrac{\startT(2i+1)-1}{2} + \left\lceil \dfrac{\ell(T)}{2} \right\rceil - i, \]
            \[ \nu_{2i+2} = \dfrac{\startT(2i+2)-2}{2} + \left\lceil \dfrac{\ell(T)}{2} \right\rceil - i. \]
            Similarly, the sequence $(\endT(1),\endT(2),\ldots,\endT(\ell(T)))$ is related to $\mu$ by
            \[ \mu_{2i+1} = \dfrac{\endT(2i+1)+1}{2} + \left\lceil \dfrac{\ell(T)}{2} \right\rceil - i, \]
            \[ \mu_{2i+2} = \dfrac{\endT(2i+2)}{2} + \left\lceil \dfrac{\ell(T)}{2} \right\rceil - i. \]
        \end{enumerate}
    \end{lemma}

    \begin{proof}
        (b) and (c) follows straight from the definition. For (a), observe that the entries in the same row (resp. column) in $T$ are exactly those on the same row (resp. column) in $T_R$ or $T_B$, so the inequalities follow.
    \end{proof}

    \begin{prop}\label{prop:cover-shuffle}
        Given partitions $\mu$ and $\nu$ with no three equal parts, let $n = \ell(\mu)$, and let $I = J = \nodd$. Then there is a bijection, denoted $\Phi$, between colored covers of wirings in $G_{\mu,\nu}$ by $(I,J)$ and shuffle tableaux of shape $\mu\oslash\nu$.
    \end{prop}

    \begin{proof}
        Colored covers of wirings in $G_{\mu,\nu}$ by $(I,J)$ biject with pairs of two families of noncrossing paths, colored red and blue. The red family gives a SSYT $T_R$ of shape $\mu_R/\nu_R$ where
        \[ \mu_R = (\mu_{2i-1}+i-k-1)_{i = 1}^k, \]
        \[ \nu_R = (\nu_{2i-1}+i-k-1)_{i = 1}^k, \]
        where $k = \left\lfloor\dfrac{n+1}{2}\right\rfloor$. The blue family gives a SSYT $T_B$ of shape $\mu_B/\nu_B$ where
        \[ \mu_B = (\mu_{2i}+i-k-1)_{i = 1}^k, \]
        \[ \nu_B = (\nu_{2i}+i-k-1)_{i = 1}^k, \]
        where $k = \left\lfloor\dfrac{n}{2}\right\rfloor$. The shuffle tableau of $T_R$ and $T_B$ is a shuffle tableau of shape $\mu\oslash\nu$.
    \end{proof}
    
    \begin{example}\label{exp:red-blue-tableau}
        Figure \ref{fig:cover-SSYT} shows a colored cover of a wiring when $\mu = (10,10,8,7,5,3)$ and $\nu = (8,5,3,2,1,1)$.
    
        \begin{figure}[h!]
            \centering
            \includegraphics[scale = 0.8]{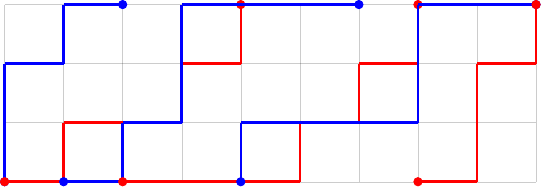}
            \caption{A colored cover}
            \label{fig:cover-SSYT}
        \end{figure}
        
        The corresponding SSYTs $T_R$ and $T_B$ of this colored cover are those in Example \ref{exp:shuffle-tableau}.
    \end{example}
    
    Next, we will define the notion of Temperley-Lieb type for shuffle tableaux. We will then show that this agrees with the Temperley-Lieb type of colored covers. That is, the Temperley-Lieb type of a colored cover is the same as that of the corresponding shuffle tableau. From the shuffle tableau, we obtain the Temperley-Lieb type as follows: for each pair
    \[ \begin{ytableau} 
    \none & j \\ 
    i & \none
    \end{ytableau}, \]
    if $i \leq j$, then draw two lines
    \[ \begin{tikzcd}[sep = tiny]
	{} & j \\
	i & {}
	\arrow[no head, from=2-1, to=1-1]
	\arrow[no head, from=1-2, to=2-2]
    \end{tikzcd}; \]
    if $i > j$, then draw two lines
    \[ \begin{tikzcd}[sep = tiny]
	{} & j \\
	i & {}
	\arrow[no head, from=2-1, to=2-2]
	\arrow[no head, from=1-2, to=1-1]
    \end{tikzcd}. \]
    By convention, the squares North and West of the tableau have value $-\infty$, and those South and East of the tableau have value $\infty$. Finally, we put $L_i$ at the end of the $i$th row and $R_i$ at the beginning of the $i$th row.
    
    \begin{example}\label{exp:tableau-diagram}
        From the shuffle tableau in Example \ref{exp:shuffle-tableau}, we have the following diagram.
        \[\begin{tikzcd}[sep=tiny]
    	&&&&&&&&&& {R_1} & 1 & {} & 3 & {} & {L_1} \\
    	&&&&& {R_2} & 2 & {} & 2 & {} & 2 & {} & 4 & {} & 4 & {L_2} \\
    	&& {R_3} & 1 & {} & 1 & {} & 1 & {} & 2 & {} & 3 & {L_3} \\
    	& {R_4} & 1 & {} & 2 & {} & 4 & {} & 4 & {} & 4 & {L_4} \\
    	{R_5} & 1 & {} & 2 & {} & 2 & {} & 3 & {L_5} \\
    	{R_6} & {} & 3 & {} & 4 & {L_6}
    	\arrow[no head, from=6-3, to=6-4]
    	\arrow[no head, from=5-4, to=5-3]
    	\arrow[no head, from=6-5, to=6-6]
    	\arrow[no head, from=5-6, to=5-5]
    	\arrow[no head, from=6-4, to=6-5]
    	\arrow[no head, from=5-2, to=4-2]
    	\arrow[no head, from=4-3, to=5-3]
    	\arrow[no head, from=5-4, to=4-4]
    	\arrow[no head, from=4-5, to=5-5]
    	\arrow[no head, from=5-6, to=4-6]
    	\arrow[no head, from=4-7, to=5-7]
    	\arrow[no head, from=4-9, to=5-9]
    	\arrow[no head, from=4-8, to=5-8]
    	\arrow[no head, from=5-7, to=5-8]
    	\arrow[no head, from=4-11, to=4-10]
    	\arrow[no head, from=4-3, to=3-3]
    	\arrow[no head, from=3-4, to=4-4]
    	\arrow[no head, from=5-2, to=6-2]
    	\arrow[no head, from=5-1, to=6-1]
    	\arrow[no head, from=6-2, to=6-3]
    	\arrow[no head, from=4-5, to=4-6]
    	\arrow[no head, from=3-6, to=3-5]
    	\arrow[no head, from=4-7, to=4-8]
    	\arrow[no head, from=3-8, to=3-7]
    	\arrow[no head, from=4-9, to=4-10]
    	\arrow[no head, from=3-10, to=3-9]
    	\arrow[no head, from=4-11, to=4-12]
    	\arrow[no head, from=3-12, to=3-11]
    	\arrow[no head, from=3-6, to=2-6]
    	\arrow[no head, from=3-4, to=3-5]
    	\arrow[no head, from=2-7, to=3-7]
    	\arrow[no head, from=3-8, to=2-8]
    	\arrow[no head, from=2-9, to=3-9]
    	\arrow[no head, from=2-11, to=3-11]
    	\arrow[no head, from=3-12, to=2-12]
    	\arrow[no head, from=2-13, to=3-13]
    	\arrow[no head, from=2-11, to=2-12]
    	\arrow[no head, from=1-12, to=1-11]
    	\arrow[no head, from=3-10, to=2-10]
    	\arrow[no head, from=2-9, to=2-10]
    	\arrow[no head, from=2-7, to=2-8]
    	\arrow[no head, from=2-15, to=1-15]
    	\arrow[no head, from=2-13, to=2-14]
    	\arrow[no head, from=1-14, to=1-13]
    	\arrow[no head, from=1-16, to=2-16]
    	\arrow[no head, from=1-14, to=1-15]
    	\arrow[no head, from=2-14, to=2-15]
    	\arrow[no head, from=1-13, to=1-12]
        \end{tikzcd}\]
        This means that the Temperley-Lieb type is
        \[ \resizebox{!}{0.25\textwidth}{
            \begin{tikzpicture}
                \draw (0,0) node[anchor=center] {$L_1$};
                \draw (0,-1) node[anchor=center] {$L_2$};
                \draw (0,-2) node[anchor=center] {$L_3$};
                \draw (0,-3) node[anchor=center] {$L_4$};
                \draw (0,-4) node[anchor=center] {$L_5$};
                \draw (0,-5) node[anchor=center] {$L_6$};
                \draw (3,0) node[anchor=center] {$R_1$};
                \draw (3,-1) node[anchor=center] {$R_2$};
                \draw (3,-2) node[anchor=center] {$R_3$};
                \draw (3,-3) node[anchor=center] {$R_4$};
                \draw (3,-4) node[anchor=center] {$R_5$};
                \draw (3,-5) node[anchor=center] {$R_6$};
                \filldraw[black] (0.5,0) circle (2pt);
                \filldraw[black] (0.5,-1) circle (2pt);
                \filldraw[black] (0.5,-2) circle (2pt);
                \filldraw[black] (0.5,-3) circle (2pt);
                \filldraw[black] (0.5,-4) circle (2pt);
                \filldraw[black] (0.5,-5) circle (2pt);
                \filldraw[black] (2.5,0) circle (2pt);
                \filldraw[black] (2.5,-1) circle (2pt);
                \filldraw[black] (2.5,-2) circle (2pt);
                \filldraw[black] (2.5,-3) circle (2pt);
                \filldraw[black] (2.5,-4) circle (2pt);
                \filldraw[black] (2.5,-5) circle (2pt);
                \draw[thick] [bend right = 90, looseness=1.25] (0.5, -1) to (0.5, 0);
                \draw[thick] [bend right = 90, looseness=1.25] (1.5, 0) to (1.5, -1);
                \draw[thick] [bend right = 90, looseness=1.25] (0.5, -4) to (0.5, -3);
                \draw[thick] [bend right = 90, looseness=1.25] (1.5, -3) to (1.5, -4);
                \draw[thick] (0.5,-2) to (1.5,-2);
                \draw[thick] (0.5,-5) to (1.5,-5);
                
                \draw[thick] [bend right = 90, looseness=1.25] (1.5, -2) to (1.5, -1);
                \draw[thick] [bend right = 90, looseness=1.25] (2.5, -1) to (2.5, -2);
                \draw[thick] [bend right = 90, looseness=1.25] (1.5, -5) to (1.5, -4);
                \draw[thick] [bend right = 90, looseness=1.25] (2.5, -4) to (2.5, -5);
                \draw[thick] (1.5,0) to (2.5,0);
                \draw[thick] (1.5,-3) to (2.5,-3);
            \end{tikzpicture}
        }. \]
    \end{example}

    We denote the Temperley-Lieb type of a shuffle tableau $T$ by $\psi(T)$. One can check that the Temperley-Lieb type in Example \ref{exp:tableau-diagram} is the same as the type of the wiring in Figure \ref{fig:cover-SSYT}.

    \begin{prop}\label{prop:correct-type}
        Let $x$ be a colored cover of a wiring $H$ of $G_{\mu,\nu}$, and let $T = \Phi(x)$ be the corresponding shuffle tableau. Then $x$ and $T$ have the same Temperley-Lieb type, i.e.
        \[ \psi(x) =  \psi(T). \]
    \end{prop}

    \begin{proof}
        We first prove the statement for \textit{canonical colored covers} by $I=J= \nodd$. Recall from Definition \ref{def:canonical-cc} that this colored cover is obtained by treating each $\resizebox{!}{0.8em}{\begin{tikzpicture}
            \cross{1}{1}{1}
        \end{tikzpicture}}$ as $\resizebox{!}{0.8em}{\begin{tikzpicture}
            \draw[thick] [bend right = 90, looseness=1.25] (-0.5, 0) to (0.5, 0);
            \draw[thick] [bend right = 90, looseness=1.25] (0.5, -1) to (-0.5, -1);
        \end{tikzpicture}}$. Figure \ref{fig:canonical-cover} shows a canonical colored cover.
    
        \begin{figure}[h!]
            \centering
            \includegraphics[scale = 0.8]{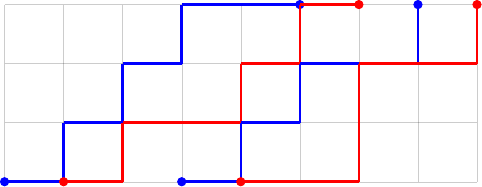}
            \caption{A canonical colored cover}
            \label{fig:canonical-cover}
        \end{figure}
        
        The shuffle tableau of the colored cover in Figure \ref{fig:canonical-cover} is
        \[ \begin{ytableau} 
        \none & \none & \none & \none & \none & 1 & \none & 1 & \none & 3 & \none & 3 \\        
        \none & \none & \none & \none & 1 & \none & 2 & \none & 3 & \none & 3 \\ 
        \none & 1 & \none & 2 & \none & 2 & \none & 3 & \none & 4 \\ 
        1 & \none & 2 & \none & 3 & \none & 4 & \none & 4
        \end{ytableau}. \]
        Recall also that in a canonical colored cover, if a line is drawn through the wiring, then it intersects paths $p_1,p_2,\ldots,p_n$ in that order. Hence, in the shuffle tableau, each pair
        \[ \begin{ytableau} 
        \none & j \\ 
        i & \none
        \end{ytableau}, \]
        with $i \leq j$ corresponds exactly with the intersections between two consecutive lines. Thus, drawing 
        \[ \begin{tikzcd}[sep = tiny]
    	{} & j \\
    	i & {}
    	\arrow[no head, from=2-1, to=1-1]
    	\arrow[no head, from=1-2, to=2-2]
        \end{tikzcd} \]
        when $i \leq j$ corresponds with replacing each intersection $\resizebox{!}{0.8em}{\begin{tikzpicture}
            \cross{1}{1}{1}
        \end{tikzpicture}}$ by $\resizebox{!}{0.8em}{\begin{tikzpicture}
             \uncross{1}{1}{1}
        \end{tikzpicture}}$.
        Hence, the resulting diagram on the shuffle tableau is the same as the diagram on the wiring, with the exception that each double edge is now a loop. This does not change the Temperley-Lieb type, so this gives the correct Temperley-Lieb type. For example, the diagram obtained from the above shuffle tableau is
        \[\begin{tikzcd}[sep=tiny]
    	&&&&& {R_1} & 1 & {} & 1 & {} & 3 & {} & 3 & {L_1} \\
    	&&&& {R_2} & 1 & {} & 2 & {} & 3 & {} & 3 & {L_2} \\
    	& {R_3} & 1 & {} & 2 & {} & 2 & {} & 3 & {} & 4 & {L_3} \\
    	{R_4} & 1 & {} & 2 & {} & 3 & {} & 4 & {} & 4 & {L_4}
    	\arrow[no head, from=1-13, to=1-14]
    	\arrow[no head, from=2-12, to=1-12]
    	\arrow[no head, from=1-13, to=2-13]
    	\arrow[no head, from=1-11, to=1-12]
    	\arrow[no head, from=1-9, to=1-10]
    	\arrow[no head, from=1-7, to=1-8]
    	\arrow[no head, from=2-8, to=2-9]
    	\arrow[no head, from=1-8, to=1-9]
    	\arrow[no head, from=2-6, to=1-6]
    	\arrow[no head, from=1-7, to=2-7]
    	\arrow[no head, from=3-11, to=3-12]
    	\arrow[no head, from=2-12, to=2-11]
    	\arrow[no head, from=3-9, to=2-9]
    	\arrow[no head, from=2-10, to=3-10]
    	\arrow[no head, from=3-7, to=2-7]
    	\arrow[no head, from=2-8, to=3-8]
    	\arrow[no head, from=3-5, to=3-6]
    	\arrow[no head, from=2-6, to=2-5]
    	\arrow[no head, from=1-10, to=2-10]
    	\arrow[no head, from=1-11, to=2-11]
    	\arrow[no head, from=3-3, to=3-4]
    	\arrow[no head, from=4-10, to=3-10]
    	\arrow[no head, from=3-11, to=4-11]
    	\arrow[no head, from=3-9, to=3-8]
    	\arrow[no head, from=4-6, to=4-7]
    	\arrow[no head, from=3-7, to=3-6]
    	\arrow[no head, from=4-4, to=3-4]
    	\arrow[no head, from=3-5, to=4-5]
    	\arrow[no head, from=4-2, to=3-2]
    	\arrow[no head, from=3-3, to=4-3]
    	\arrow[no head, from=4-1, to=4-2]
    	\arrow[no head, from=4-3, to=4-4]
    	\arrow[no head, from=4-5, to=4-6]
    	\arrow[no head, from=4-7, to=4-8]
    	\arrow[no head, from=4-9, to=4-10]
    	\arrow[no head, from=4-8, to=4-9]
        \end{tikzcd},\]
        which is the same as the diagram for the wiring in Figure \ref{fig:canonical-cover}.

        Finally, observe that every colored cover can be obtained from the canonical colored cover by swapping the orientation of some loops. In terms of the shuffle tableau, this means permuting the entries on each diagonal of the loop. For example, swapping the orientation of the biggest loop in Figure \ref{fig:canonical-cover} gives the colored cover in Figure \ref{fig:not-canonical-cover}.

        \begin{figure}[h!]
            \centering
            \includegraphics[scale = 0.8]{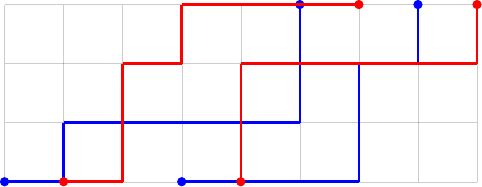}
            \caption{Another colored cover}
            \label{fig:not-canonical-cover}
        \end{figure}

        The diagram of the shuffle tableau corresponding to Figure \ref{fig:not-canonical-cover} is
        \[\begin{tikzcd}[sep=tiny]
    	&&&&& {R_1} & 3 & {} & 3 & {} & 3 & {} & 3 & {L_1} \\
    	&&&& {R_2} & 1 & {} & 1 & {} & 1 & {} & 3 & {L_2} \\
    	& {R_3} & 1 & {} & 3 & {} & 4 & {} & 4 & {} & 4 & {L_3} \\
    	{R_4} & 1 & {} & 2 & {} & 2 & {} & 2 & {} & 2 & {L_4}
    	\arrow[no head, from=1-13, to=1-14]
    	\arrow[no head, from=2-12, to=1-12]
    	\arrow[no head, from=1-13, to=2-13]
    	\arrow[no head, from=1-11, to=1-12]
    	\arrow[no head, from=1-9, to=1-10]
    	\arrow[no head, from=1-7, to=1-8]
    	\arrow[no head, from=2-6, to=1-6]
    	\arrow[no head, from=1-7, to=2-7]
    	\arrow[no head, from=3-11, to=3-12]
    	\arrow[no head, from=2-12, to=2-11]
    	\arrow[no head, from=3-5, to=3-6]
    	\arrow[no head, from=2-6, to=2-5]
    	\arrow[no head, from=1-10, to=2-10]
    	\arrow[no head, from=1-11, to=2-11]
    	\arrow[no head, from=3-3, to=3-4]
    	\arrow[no head, from=4-10, to=3-10]
    	\arrow[no head, from=3-11, to=4-11]
    	\arrow[no head, from=4-4, to=3-4]
    	\arrow[no head, from=3-5, to=4-5]
    	\arrow[no head, from=4-2, to=3-2]
    	\arrow[no head, from=3-3, to=4-3]
    	\arrow[no head, from=4-1, to=4-2]
    	\arrow[no head, from=4-3, to=4-4]
    	\arrow[no head, from=4-5, to=4-6]
    	\arrow[no head, from=4-7, to=4-8]
    	\arrow[no head, from=4-9, to=4-10]
    	\arrow[no head, from=1-8, to=2-8]
    	\arrow[no head, from=1-9, to=2-9]
    	\arrow[no head, from=3-9, to=3-10]
    	\arrow[no head, from=2-9, to=2-10]
    	\arrow[no head, from=3-7, to=3-8]
    	\arrow[no head, from=2-7, to=2-8]
    	\arrow[no head, from=3-8, to=4-8]
    	\arrow[no head, from=3-9, to=4-9]
    	\arrow[no head, from=3-7, to=4-7]
    	\arrow[no head, from=3-6, to=4-6]
        \end{tikzcd},\]
        where the entries on each diagonal of the loop are permuted. However, in a canonical colored cover, if a line is drawn through the wiring, then it intersects paths $p_1,p_2,\ldots,p_n$ in that order. This means that in the shuffle tableau of the canonical colored cover, the entries on each diagonal are weakly increasing, with each equality corresponds to a double edge. This means that after permuting the entries on the diagonals, the entries on the South and East borders of the loop do not increase, and those on the North and West borders do not decrease. This means that the interactions between border entries and outside entries do not change, and the Temperley-Lieb type does not change.
    \end{proof}

    \begin{remark}\label{rem:cycle}
        Recall from Remark \ref{rem:loop-and-order} that if for some colored cover of a wiring $H$ by $I = J$, a vertical line drawn through the wiring intersects path $p_i$ before $p_j$ where $i > j$, then the segments of $p_i$ and $p_j$ that intersect the line have to belong to some loops. In terms of the shuffle tableau, this implies that if there is a pair
        \[ \begin{ytableau} 
        j & \none \\ 
        \none & i
        \end{ytableau} \]
        where $j \geq i$, then both squares belong to some loops in the tableau.
    \end{remark}

    Proposition \ref{prop:correct-type} means that the bijection between shuffle tableaux of shape $\mu\oslash\nu$ and colored covers of $G_{\mu,\nu}$ preserves the Temperley-Lieb type. Let $\omega(T)$ be the weight of $T$, then Corollary \ref{cor:colored-cover} is equivalent to the following corollary.

    \begin{cor}\label{cor:shuffle-tableau}
        For any basis element $\tau$ of $TL_n(2)$ and partitions $\mu,\nu$, we have
        \[ \Imm^{\TL}_\tau(A_{\mu,\nu}) = \sum_{\substack{\text{$T$ of shape $\mu\oslash\nu$} \\ \psi(T) = \tau}} \omega(T). \]
    \end{cor}


\section{Crystal operators}\label{sec:crys}


    Given a shuffle tableau $T$ of shape $\mu\oslash\nu$, an $(i,i+1)$-overlap is a pair of squares $(s,t)$ such that $s$ contains an $i$, $t$ contains an $i+1$, and $s$ and $t$ are on the same column. If $(s,t)$ is an $(i,i+1)$-overlap pair, we say that $s$ and $t$ are $(i,i+1)$-overlapped. 
    
    \begin{example}\label{exp:overlap}
        In the following tableau $T$,
        \[ \begin{ytableau}    
        \none & \none & \none & \none & \textcolor{red}{1} & \none & \textcolor{red}{1} & \none & \textcolor{red}{1} & \none & 1 & \none & 2 \\ 
        \none & \none & \none & \textcolor{red}{1} & \none & \textcolor{blue}{2} & \none & 3 & \none & 3 & \none & 3 \\ 
        \none & \none & \textcolor{red}{1} & \none & \textcolor{red}{2} & \none & \textcolor{red}{2} & \none & \textcolor{red}{2} & \none & 3 \\ 
        \none & 2 & \none & \textcolor{red}{2} & \none & \textcolor{blue}{3} \\ 
        2 & \none & \textcolor{red}{2}
        \end{ytableau}, \]
        there are five pairs $(1,2)$-overlap (in red), and one pair of $(2,3)$-overlap (in blue).
    \end{example}

    Next, we define the $i$-reading word $w_i(T)$ as follows.
    \begin{enumerate}
        \item Consider the squares consisting of $i$ and $i+1$.
        \item Remove all $(i,i+1)$-overlap pairs.
        \item Iteratively read the remaining squares from bottom to top, left to right.
    \end{enumerate}

    \begin{example}\label{exp:reading-word}
        In the tableau in Example \ref{exp:overlap}, we have
        \[ w_1(T) = 2~|~2~|~|~2~|~1~2 \]
        by reading all non-red $1$s and $2$s, and
        \[ w_2(T) = 2~2~|~2~2~|~2~2~2~3~|~3~3~3~|~2 \]
        by reading all non-blue $2$s and $3$s.
    \end{example}

    In addition, we say square $(r_1,c_1)$ is \textit{right of} square $(r_2,c_2)$ if $r_2 > r_1$ or $r_2 = r_1$ and $c_2 < c_1$. Then, we also say $(r_2,c_2)$ is \textit{left of} $(r_1,c_1)$.

    Finally, we define the crystal operators on $T$ by acting on $w_i(T)$ in the same fashion as classical crystal operators acting on words.

    \begin{definition}[Crystal operators]\label{def:crys-ops}
        The \textit{crystal operators} $E_i$ and $F_i$ act on $w_i(T)$ in the following (standard) way. 
        \begin{itemize}
            \item View each $i$ as a closing parenthesis ``)'' and each $i+1$ as an opening parenthesis ``(''.
            \item Match the parentheses in the usual way.
            \item $E_i$ changes the leftmost unmatched $i+1$ to $i$, and $F_i$ changes the rightmost unmatched $i$ to $i+1$.
        \end{itemize}
   \end{definition}
     
       We induce this action of the crystal operators to the action on shuffle tableaux as follows: 
        \begin{itemize}
            \item the shape of the shuffle tableaux is preserved;
            \item The content changes in the way uniquely determined by $w_i(E_i(T)) = E_i(w_i(T))$ and $w_i(F_i(T)) = F_i(w_i(T))$.
        \end{itemize} 
        
     \begin{prop}\label{prop:well-defined}
         This action is well-defined.
     \end{prop}
     
     \begin{proof}
        For $E_i(T)$ to be a valid shuffle tableaux, we need to check that the leftmost unmatched $i+1$ in $w_i(T)$ corresponds to the leftmost $(i+1)$ on some row in $T$. If this is not the case, then let $s$ be the square corresponding to the leftmost unmatched $i+1$ in $w_i(T)$, and let $s'$ be the square containing the leftmost $(i+1)$ on the same row as $s$. Since $s$ is not $(i,i+1)$-overlapped, $s'$ is also not $(i,i+1)$-overlapped. Hence, $s'$ is matched in $w_i(T)$. However, the squares from $s'$ to $s$ form a sequence on consecutive $(i+1)$'s in $w_i(T)$, so it is impossible for $s'$ to be matched when $s$ is unmatched. This proves that in $E_i(T)$, the entries are weakly increasing along the rows.

        Furthermore, since $E_i$ only affects non-$(i,i+1)$-overlapped $(i+1)$-edges, so in $E_i(T)$, the entries are still strictly increasing along the columns.

        The argument for $F_i$ is analogous.
    \end{proof}
    \begin{definition}[Temperley-Lieb crystals]\label{def:TLcyrs}
        A \textit{Temperley-Lieb crystal} is a connected component of the graph on shuffle tableaux formed by the crystal operators.
    \end{definition}

    In the remainder of this paper, we will prove that the tableaux in the same Temperley-Lieb crystal have the same Temperley-Lieb type, and that the Temperley-Lieb crystals satisfy type $A$ Stembridge's axioms.

    \begin{example}\label{exp:crystals}
        For $\mu = (4,3,3)$ and $\nu = (3,2,1)$, the Jacobi-Trudi matrix is
        \[ A_{\mu,\nu} = \left(\begin{matrix}
            h_1 & h_2 & h_3 \\
            1 & h_1 & h_2 \\
            1 & h_1 & h_2 \\
        \end{matrix}\right) . \]
        Let $I = J = \{1,3\}$, then
        \[ \Delta_{I,J}(A_{\mu,\nu})\cdot\Delta_{\Bar{I},\Bar{J}}(A_{\mu,\nu}) = \left(\begin{matrix}
            h_1 & h_3 \\
            1 & h_2 \\
        \end{matrix}\right) \cdot h_1 = s_{(3,1)} + s_{(2,2)} + s_{(2,1,1)}. \]
        In particular, let $\tau_1$ be
        \[ \resizebox{!}{0.1\textwidth}{
            \begin{tikzpicture}
                \draw (0,0) node[anchor=center] {$L_1$};
                \draw (0,-1) node[anchor=center] {$L_2$};
                \draw (0,-2) node[anchor=center] {$L_3$};
                \draw (2,0) node[anchor=center] {$R_1$};
                \draw (2,-1) node[anchor=center] {$R_2$};
                \draw (2,-2) node[anchor=center] {$R_3$};
                \filldraw[black] (0.5,0) circle (2pt);
                \filldraw[black] (0.5,-1) circle (2pt);
                \filldraw[black] (0.5,-2) circle (2pt);
                \filldraw[black] (1.5,0) circle (2pt);
                \filldraw[black] (1.5,-1) circle (2pt);
                \filldraw[black] (1.5,-2) circle (2pt);
                \uncross{1}{2}{2}
            \end{tikzpicture}
        }, \]
        then
        \[ \Imm^{\TL}_{\tau_1}(A_{\mu,\nu}) = s_{(3,1)} + s_{(2,1,1)}. \]
        Meanwhile, let $\tau_2$ be
        \[ \resizebox{!}{0.1\textwidth}{
            \begin{tikzpicture}
                \draw (0,0) node[anchor=center] {$L_1$};
                \draw (0,-1) node[anchor=center] {$L_2$};
                \draw (0,-2) node[anchor=center] {$L_3$};
                \draw (3,0) node[anchor=center] {$R_1$};
                \draw (3,-1) node[anchor=center] {$R_2$};
                \draw (3,-2) node[anchor=center] {$R_3$};
                \filldraw[black] (0.5,0) circle (2pt);
                \filldraw[black] (0.5,-1) circle (2pt);
                \filldraw[black] (0.5,-2) circle (2pt);
                \filldraw[black] (2.5,0) circle (2pt);
                \filldraw[black] (2.5,-1) circle (2pt);
                \filldraw[black] (2.5,-2) circle (2pt);
                \uncross{1}{2}{2}
                \uncross{2}{1}{2}
            \end{tikzpicture}
        }, \]
        then
        \[ \Imm^{\TL}_{\tau_2}(A_{\mu,\nu}) = s_{(2,2)}. \]
        
        This means that the crystal operators should group the shuffle tableaux of shape $\mu \oslash \nu$ into three crystals corresponding to the three Schur polynomials. In addition, all tableaux in the same crystal should have the correct Temperley-Lieb type, e.g. all tableaux in the crystal corresponding to $s_{(3,1,1)}$ should have type $\tau_1$. Indeed, the three crystals are shown in Figure \ref{fig:crystals}, in which the green arrows represent $F_1$, and the violet arrows represent $F_2$. The readers are encouraged to check that all tableaux in the same crystal have the correct Temperley-Lieb type.

        \begin{figure}[h!]
         \centering
            \begin{subfigure}[b]{0.4\textwidth}
                \centering
                \includegraphics[scale = 0.8]{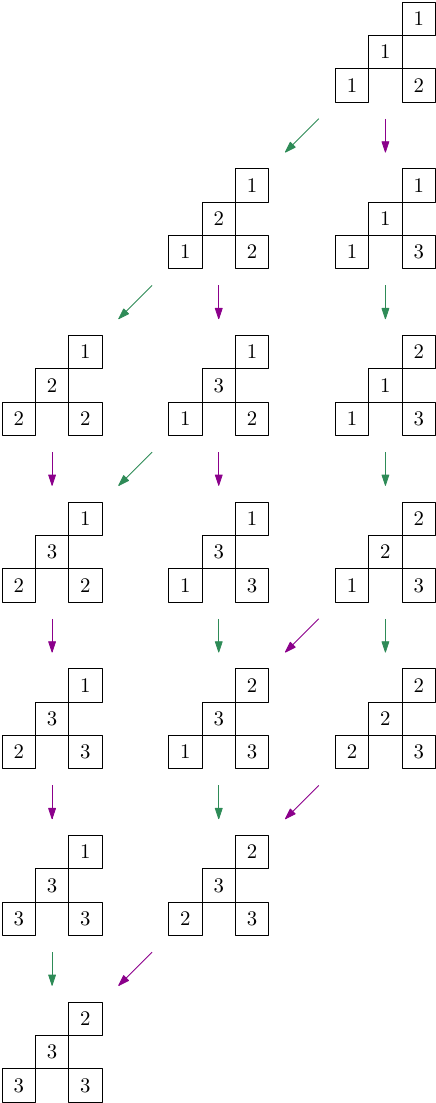}
                \caption{$s_{(3,1)}$}
                \label{subfig:s31}
            \end{subfigure}
         \quad
            \begin{subfigure}[b]{0.25\textwidth}
                \centering
                \includegraphics[scale = 0.8]{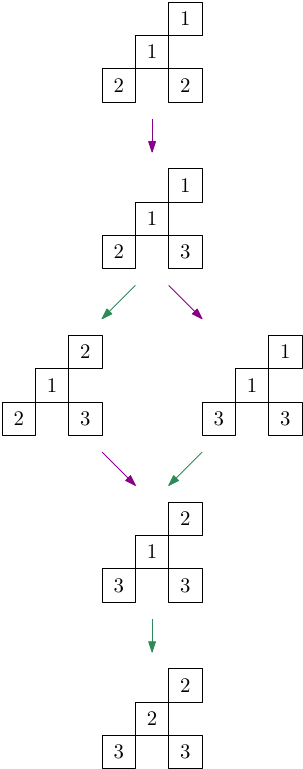}
                \caption{$s_{(2,2)}$}
                \label{subfig:s22}
            \end{subfigure}
         \quad
            \begin{subfigure}[b]{0.2\textwidth}
                \centering
                \includegraphics[scale = 0.8]{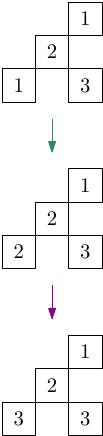}
                \caption{$s_{(2,1,1)}$}
                \label{subfig:s211}
            \end{subfigure}

            \caption{Three TL crystals and the corresponding generating functions}
            \label{fig:crystals}
        \end{figure}

        For example, we can quickly check the Temperley-Lieb type of the highest tableau of each crystal as follows
        \[\begin{tikzcd}[sep = tiny]
    	&& {R_1} & 1 & {L_1} \\
    	& {R_2} & 1 & {} & {L_2} \\
    	{R_3} & 1 & {} & 2 & {L_3}
    	\arrow[no head, from=1-3, to=2-3]
    	\arrow[no head, from=1-4, to=2-4]
    	\arrow[no head, from=1-4, to=1-5]
    	\arrow[no head, from=2-5, to=3-5]
    	\arrow[no head, from=3-2, to=2-2]
    	\arrow[no head, from=2-3, to=3-3]
    	\arrow[no head, from=3-1, to=3-2]
    	\arrow[no head, from=3-3, to=3-4]
    	\arrow[no head, from=3-4, to=2-4]
        \end{tikzcd},\]
        \[\begin{tikzcd}[sep = tiny]
    	&& {R_1} & 1 & {L_1} \\
    	& {R_2} & 1 & {} & {L_2} \\
    	{R_3} & 2 & {} & 2 & {L_3}
    	\arrow[no head, from=1-3, to=2-3]
    	\arrow[no head, from=1-4, to=2-4]
    	\arrow[no head, from=1-4, to=1-5]
    	\arrow[no head, from=2-5, to=3-5]
    	\arrow[no head, from=3-1, to=3-2]
    	\arrow[no head, from=3-3, to=3-4]
    	\arrow[no head, from=3-4, to=2-4]
    	\arrow[no head, from=3-2, to=3-3]
    	\arrow[no head, from=2-3, to=2-2]
        \end{tikzcd},\]
        \[\begin{tikzcd}[sep = tiny]
    	&& {R_1} & 1 & {L_1} \\
    	& {R_2} & 2 & {} & {L_2} \\
    	{R_3} & 1 & {} & 3 & {L_3}
    	\arrow[no head, from=1-4, to=1-5]
    	\arrow[no head, from=2-5, to=3-5]
    	\arrow[no head, from=3-1, to=3-2]
    	\arrow[no head, from=3-3, to=3-4]
    	\arrow[no head, from=3-4, to=2-4]
    	\arrow[no head, from=2-2, to=3-2]
    	\arrow[no head, from=2-3, to=3-3]
    	\arrow[no head, from=1-3, to=1-4]
    	\arrow[no head, from=2-3, to=2-4]
        \end{tikzcd}.\]
    \end{example}

    \begin{prop}\label{prop:same-type}
        $E_i$ and $F_i$ do not change the Temperley-Lieb type. That is,
        \[ \psi(T) = \psi(E_iT) = \psi(F_iT). \]
    \end{prop}

    \begin{proof}
        We will prove the statement for $E_i$. The statement for $F_i$ follows.

        Suppose $E_i$ acts on $T$ by changing square $(r,c)$ from $i+1$ to $i$. The diagram of $T$ only changes if either $(r+1,c-1)$ contains $i+1$ or $(r-1, c+1)$ contains $i$. We consider the case where $(r+1,c-1)$ contains $i+1$. The argument for when $(r-1,c+1)$ contains $i$ is dual to this case. In this case, the diagram is changed from
        \[ \begin{tikzcd}[sep = tiny]
    	{} & i+1 \\
    	i+1 & {}
    	\arrow[no head, from=2-1, to=1-1]
    	\arrow[no head, from=1-2, to=2-2]
        \end{tikzcd} \]
        to
        \[ \begin{tikzcd}[sep = tiny]
    	{} & i \\
    	i+1 & {}
    	\arrow[no head, from=2-1, to=2-2]
    	\arrow[no head, from=1-2, to=1-1]
        \end{tikzcd}. \]
        Recall from Remark \ref{rem:cycle} that if $(r+1,c+1)$ contains $i+1$, then both $(r+1,c+1)$ and $(r,c)$ belong to some loop in the diagram of $T$, so this change does not affect the Temperley-Lieb type. Thus, we only need to consider the case where $(r+1,c+1)$ contains some $x > i+1$. In addition, since $(r+1,c-1)$ contains $i+1$, $(r-1,c-1)$ is at most $i$. We have two cases.

        \textit{Case 1:} $(r-1,c-1)$ contains $i$. Let $a$ be the number of non-$(i,i+1)$-overlapped $i$-squares on row $r-1$. Observe that $(r-1,c-1)$ is $(i,i+1)$-overlapped, but $(r-1,c+1)$ is not $(i,i+1)$-overlapped because $(r+1,c+1)$ contains some $x > i+1$. Hence, the non-$(i,i+1)$-overlapped $i$'s on row $r-1$ are in squares $(r-1,c+1),(r-1,c+3),\ldots,(r-1,c+2a-1)$, and square $(r-1,c+2a+1)$ contains some $z \geq i+1$. In addition, square $(r+1,c+2j-1)$ contains some $x_j > i+1$ for $1\leq j \leq a$.

        On the other hand, $E_i$ acts on $T$ at $(r,c)$, so by Proposition \ref{prop:well-defined}, $(r,c)$ is the leftmost $(i+1)$-square on row $r$ in $T$. Hence, square $(r,c-2)$ contains some $y \leq i$. Furthermore, $(r,c)$ contains an unmatched $i+1$ in $w_i(T)$, so on row $r$, there are at least $a+1$ non-$(i,i+1)$-overlapped $(i+1)$-squares in squares $(r,c),(r,c+2),\ldots,(r,c+2a)$. These $(i+1)$'s are not $(i,i+1)$-overlapped, so square $(r-2,c+2j-2)$ has to contain some $w_j < i$ for $1\leq j \leq a+1$.

        In summary, we have the following configuration, where $(r,c)$ is colored red,
        \[\begin{tikzcd}[sep=tiny]
    	& {} & {w_1} & {} & {w_2} & {} & \ldots & {} & {w_{a+1}} \\
    	{} & i & {} & i & {} & i & \ldots & i & {} & z \\
    	y & {} & {\textcolor{red}{i+1}} & {} & {i+1} & {} & \ldots & {} & {i+1} & {} \\
    	& {i+1} & {} & x & {} & {x_2} & \ldots & {x_a} & {}
    	\arrow[no head, from=3-3, to=4-3]
    	\arrow[no head, from=3-1, to=2-1]
    	\arrow[no head, from=2-2, to=2-3]
    	\arrow[no head, from=2-4, to=2-5]
    	\arrow[no head, from=2-6, to=2-7]
    	\arrow[no head, from=2-6, to=2-5]
    	\arrow[no head, from=2-3, to=2-4]
    	\arrow[no head, from=3-3, to=3-4]
    	\arrow[no head, from=3-4, to=3-5]
    	\arrow[no head, from=1-3, to=1-2]
    	\arrow[no head, from=1-5, to=1-4]
    	\arrow[no head, from=4-4, to=4-5]
    	\arrow[no head, from=4-6, to=4-7]
    	\arrow[no head, from=4-2, to=3-2]
    	\arrow[no head, from=3-2, to=2-2]
    	\arrow[no head, from=2-8, to=2-9]
    	\arrow[no head, from=1-9, to=1-8]
    	\arrow[no head, from=2-7, to=2-8]
    	\arrow[no head, from=3-9, to=2-9]
    	\arrow[no head, from=2-10, to=3-10]
    	\arrow[no head, from=3-9, to=3-8]
    	\arrow[no head, from=4-8, to=4-9]
    	\arrow[no head, from=3-8, to=3-7]
    	\arrow[no head, from=3-5, to=3-6]
    	\arrow[no head, from=3-6, to=3-7]
    	\arrow[no head, from=1-7, to=1-6]
        \end{tikzcd}\]
        Hence, in the diagram, $(r,c)$ and $(r+1,c-1)$ belong to the same strand, so $E_i(T)$ does not change the Temperley-Lieb type.

        \textit{Case 2:} $(r-1,c-1)$ contains $w_1 < i$ (potentially $w_1 = -\infty$). Let $a$ be the number of non-$(i,i+1)$-overlapped $(i+1)$-squares on row $r+1$. $(r-1,c-1)$ containing $w_1 < i$ means that $(r+1,c-1)$ contains a non-$(i,i+1)$-overlapped $(i+1)$. Recall also that we are in the case where $(r+1,c+1)$ contains some $x > i+1$. Thus, $(r+1,c-1)$ is the rightmost $(i+1)$-square on row $r+1$. Hence, $(r+1,c-1)$ being non-$(i,i+1)$-overlapped means that all $(i+1)$-squares on row $r+1$ are non-$(i,i+1)$-overlapped. Hence, the $(i+1)$-squares on row $r+1$ are in squares $(r+1,c-1),(r+1,c-3),\ldots,(r+1,c-2a+1)$, and square $(r+1,c-2a-1)$ contains some $z \leq i$. In addition, square $(r-1,c-2j+1)$ contains some $w_j < i$ for $1\leq j \leq a$.

        On the other hand, $(r,c)$ contains the leftmost unmatched $i+1$ in $w_i(T)$, so on row $r$, there are at least $a$ non-$(i,i+1)$-overlapped $i$-squares in squares $(r,c-2),\ldots,(r,c-2a)$. These $i$'s are non-$(i,i+1)$-overlapped, so square $(r+2,c-2j)$ contains some $x_j > i+1$ for $1\leq j\leq a$.

        In summary, we have the following configuration, where $(r,c)$ is colored red,
        \[\begin{tikzcd}[sep=tiny]
    	& {} & {w_a} & {} & \ldots & {} & {w_2} & {} & {w_1} \\
    	{} & i & {} & i & \ldots & i & {} & i & {} & {\textcolor{red}{i+1}} \\
    	z & {} & {i+1} & {} & \ldots & {} & {i+1} & {} & {i+1} & {} \\
    	& {x_a} & {} & {x_{a-1}} & \ldots & {x_2} & {} & {x_1} & {}
    	\arrow[no head, from=1-3, to=1-2]
    	\arrow[no head, from=2-2, to=2-3]
    	\arrow[no head, from=2-4, to=2-5]
    	\arrow[no head, from=1-5, to=1-4]
    	\arrow[no head, from=2-6, to=2-7]
    	\arrow[no head, from=1-7, to=1-6]
    	\arrow[no head, from=2-8, to=2-9]
    	\arrow[no head, from=1-9, to=1-8]
    	\arrow[no head, from=2-9, to=3-9]
    	\arrow[no head, from=2-10, to=3-10]
    	\arrow[no head, from=4-8, to=4-9]
    	\arrow[no head, from=3-9, to=3-8]
    	\arrow[no head, from=4-6, to=4-7]
    	\arrow[no head, from=3-7, to=3-6]
    	\arrow[no head, from=4-4, to=4-5]
    	\arrow[no head, from=3-5, to=3-4]
    	\arrow[no head, from=4-2, to=4-3]
    	\arrow[no head, from=3-3, to=3-2]
    	\arrow[no head, from=3-1, to=2-1]
    	\arrow[no head, from=2-2, to=3-2]
    	\arrow[no head, from=2-3, to=2-4]
    	\arrow[no head, from=3-3, to=3-4]
    	\arrow[no head, from=2-5, to=2-6]
    	\arrow[no head, from=3-5, to=3-6]
    	\arrow[no head, from=2-7, to=2-8]
    	\arrow[no head, from=3-7, to=3-8]
        \end{tikzcd}\]
        Hence, square $(r+1,c-1)$ belongs to a loop, so $E_i(T)$ does not change the Temperley-Lieb type.
    \end{proof}

\section{Stembridge's axioms}\label{sec:axiom}

    In this section, we will prove the following theorem.

    \begin{thm}\label{thm:typeA-crystal}
        Temperley-Lieb crystals are type A Kashiwara crystals.
    \end{thm}

    In \cite{stembridge2003local} Stembridge has characterized simply-laced crystals in terms of local axioms. In other words, he formulated several conditions checking which proves that the graph is a Kashiwara crystal of a particular simply-laced Kac-Moody type. 
    We prove Theorem \ref{thm:typeA-crystal} by verifying Stembridge's axioms for TL crystals in Lemmas \ref{lem:P1-P2}, \ref{lem:P5}, \ref{lem:P6}, and \ref{lem:P5'6'}.

    Now we recall Stembridge's axioms in \cite{stembridge2003local}. Let $I$ be a finite index set an $A = [a_{ij}]_{i,j\in I}$ be the Cartan matrix of a simply-laced Kac-Moody algebra, i.e. $a_{ii} = 2$, and $a_{ij} = a_{ji} \in \{0,-1\}$ for all $i\neq j$. In our case, we are working with type $A$, so $I = [n]$, $a_{ij} = a_{ji} = -1$ if $|i-j| = 1$, and $a_{ij} = a_{ji} = 0$ if $|i-j| \geq 2$.

    Consider a directed graph $X$ whose edges are colored by labels from $I$. $X$ is $A$-regular if it satisfies the requirements (P1) - (P6) and (P5'), (P6') below.

    \begin{itemize}
        \item[(P1)] All monochromatic directed paths in $X$ have finite length. In particular $X$ has no monochromatic circuits.
        \item[(P2)] For every vertex $x\in X$ and $i\in I$, there is at most one edge $y \xrightarrow{i} x$ and at most one edge $x\xrightarrow{i} z$.
    \end{itemize}

    By (P2), we can define $y = E_i(x)$ if there is an edge $y \xrightarrow{i} x$, and $z = F_i(x)$ is there is an edge $x \xrightarrow{i} z$. We also define the $i$-string through $x$ to be the maximal path
    \[ F_i^{-d}(x) \rightarrow \ldots \rightarrow F_i^{-1}(x) \rightarrow x \rightarrow F_i(x) \rightarrow \ldots \rightarrow F_i^r(x). \]
    By (P1), we have $0 \leq r,d < \infty$. We have the following notations.
    \begin{align*}
        \delta(x,i) = -d,\quad\quad\quad\quad &\eps(x,i) := r, \\
        \Delta_i\delta(x,j) = \delta(E_ix,j) - \delta(x,j), \quad\quad\quad\quad & \Delta_i\eps(x,j) = \eps(E_ix,j) - \eps(x,j), \\
        \nabla_i\delta(x,j) = \delta(x,j) - \delta(F_ix,j), \quad\quad\quad\quad & \nabla_i\eps(x,j) = \eps(x,j) - \eps(F_ix,j).
    \end{align*}

    \begin{example}\label{exp:deltaesp}

        For example, let $x$ be the red tableau in Figure \ref{fig:deltaesp}, then the $1$-string of $x$ is the path of green arrows going through $x$. This means that $\delta(x,1) = -2$ and $\eps(x,1) = 3$.

        $E_2(x)$ is the tableau to the left of $x$, so we have $\delta(E_2x,1) = -3$. This means that $\Delta_2\delta(x,1) = \delta(E_2x,1) - \delta(x,1) = (-3) - (-2) = -1$. On the other hand, $\eps(E_2x,1) = 3$, so $\Delta_2\eps(x,1) = \eps(E_2x,1)- \eps(x,1) = 3 - 3 = 0$.

        Similarly, $F_2(x)$ is the tableau to the right of $x$, so we have $\delta(F_2x,1) = -1$. Hence, $\nabla_2\delta(x,1) = \delta(x,1) - \delta(F_2x,1) = (-2) - (-1) = -1$. On the other hand, $\eps(F_2x,1) = 3$, so $\nabla_2\eps(x,1) = \eps(x,1)- \eps(F_2x,1) = 3 - 3 = 0$.
        
        \begin{figure}[h!]
            \centering
            \includegraphics[scale = 0.75]{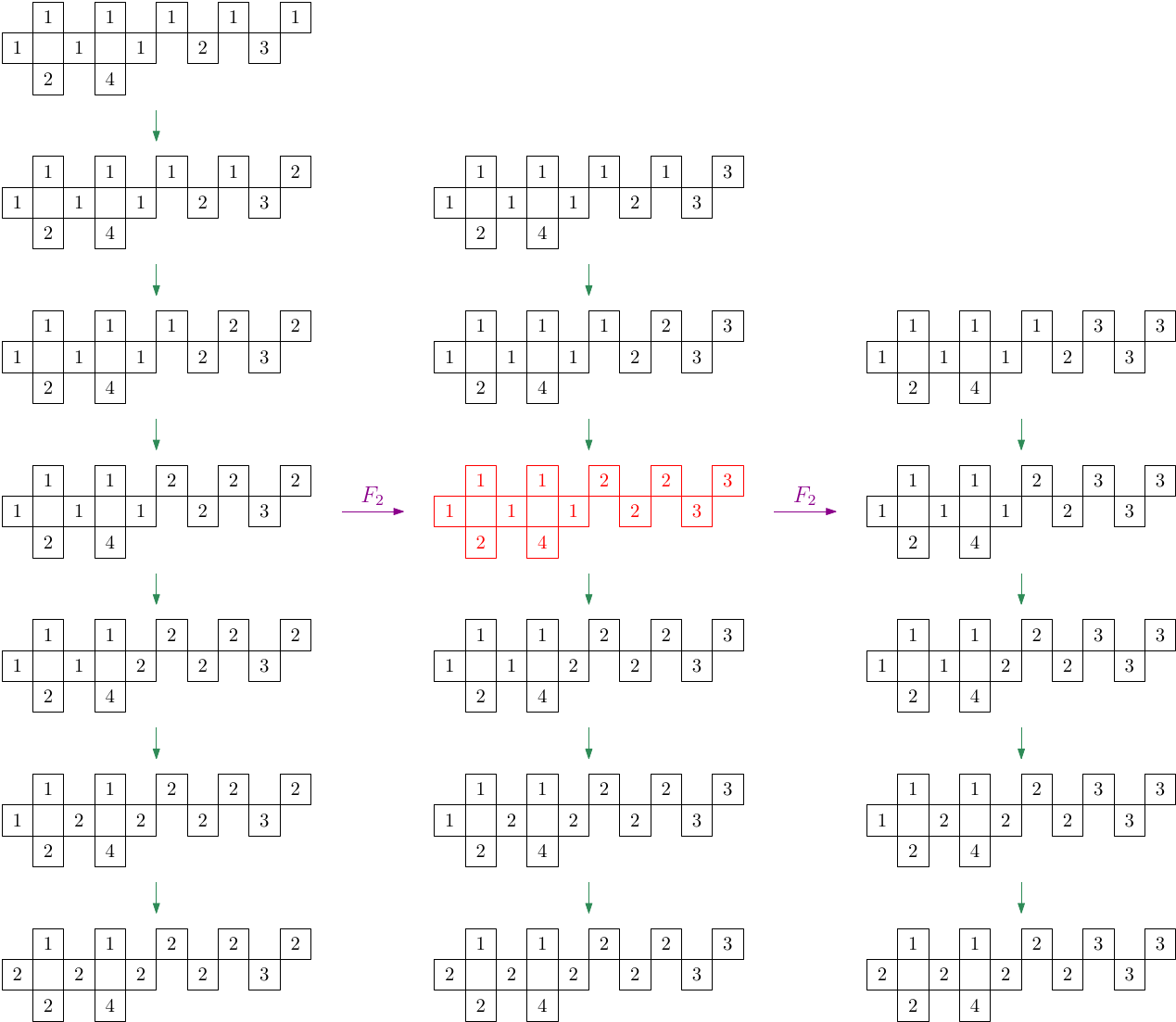}
            \caption{$1$-strings}
            \label{fig:deltaesp}
        \end{figure}
    \end{example}
    Note that that notations that involve $E_i$ (resp. $F_i$) are defined only when $E_ix$ (resp. $F_ix$) is defined. Now for any $x\in X$ and $i\neq j\in I$, when $E_ix$ is defined, we require the following two requirements.
    \begin{itemize}
        \item[(P3)] $\Delta_i\delta(x,j) + \Delta_i\eps(x,j) = a_{ij}$.
        \item[(P4)] $\Delta_i\delta(x,j) \leq 0, \Delta_i\eps(x,j) \leq 0$.
    \end{itemize}
    \begin{example}\label{exp:P3,4}
        Recall from Example \ref{exp:deltaesp} that $\Delta_2\delta(x,1) = \nabla_2\delta(x,1) = -1$, and $\Delta_2\eps(x,1) = \nabla_2\eps(x,1) = 0$, where $x$ is the red tableau in Figure \ref{fig:deltaesp}. Also, $\Delta_2\delta(x,1) + \Delta_2\eps(x,1) = -1$.
    \end{example}
    When $E_ix$ and $E_jx$ are both defined, we have the following two requirements.
    \begin{itemize}
        \item[(P5)] $\Delta_i\delta(x,j) = 0$ implies $y := E_iE_jx = E_jE_ix$, and $\nabla_j\eps(y,i) = 0$.
        \item[(P6)] $\Delta_i\delta(x,j) = \Delta_j\delta(x,i) = -1$ implies $y:= E_iE_j^2E_ix = E_jE_i^2E_jx$, and $\nabla_i\eps(y,j) = \nabla_j\eps(y,i) = -1$
    \end{itemize}
    Dually, when $F_ix$ and $F_jx$ are both defined, we have the following two requirements.
    \begin{itemize}
        \item[(P5')] $\nabla_i\eps(x,j) = 0$ implies $y := F_iF_jx = F_jF_ix$, and $\Delta_j\delta(y,i) = 0$.
        \item[(P6')] $\nabla_i\eps(x,j) = \nabla_j\eps(x,i) = -1$ implies $y:= F_iF_j^2F_ix = F_jF_i^2F_jx$, and $\Delta_i\delta(y,j) = \Delta_j\delta(y,i) = -1$
    \end{itemize}

    \begin{lemma}\label{lem:P1-P2}
        The crystal operators defined in Section \ref{sec:crys} satisfy axioms (P1) - (P4).
    \end{lemma}

    \begin{proof}
        (P1) and (P2) follow directly from the definition. In particular, $\delta(x,i)$ (resp. $\eps(x,i)$) is the number of unmatched $i+1$ (resp. $i$) in $w_i(x)$.

        (P3) and (P4) also follow from the definition when $|i-j|>1$. If $j = i+1$, then $E_i(x)$ turns an $(i+1)$ in some square $s$ to an $i$.
        
        If $s$ is $(i+1,i+2)$-overlapped in $x$, then the $(i+2)$-square that overlaps with $s$ in $x$ is not $(i+1,i+2)$-overlapped in $E_i(x)$. Thus, $w_{i+1}(E_i(x))$ is $w_{i+1}(x)$ with an extra $i+2$. If this new $i+2$ creates another pair of parentheses in $w_{i+1}(E_i(x))$, then $\Delta_i\eps(x,i+1) = -1$ and $\Delta_i\delta(x,i+1) = 0$ (see Figure \ref{subfig:P34overlapa}). Else, $\Delta_i\eps(x,i+1) = 0$ and $\Delta_i\delta(x,i+1) = -1$ (see Figure \ref{subfig:P34overlapb}). In both cases, (P3) and (P4) are satisfied.

        \begin{figure}[h!]
         \centering
            \begin{subfigure}[b]{\textwidth}
                \centering
                \includegraphics[scale = 2]{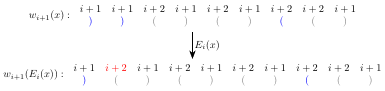}
                \caption{Extra $i+2$ (in red) creates a new pair of parentheses}
                \label{subfig:P34overlapa}
            \end{subfigure}
            
         \vspace{2em}
         
            \begin{subfigure}[b]{\textwidth}
                \centering
                \includegraphics[scale = 2]{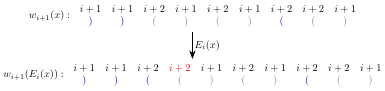}
                \caption{Extra $i+2$ (in red) does not create a new pair of parentheses}
                \label{subfig:P34overlapb}
            \end{subfigure}

            \caption{Possible scenarios with $w_{i+1}$ when $s$ is $(i+1,i+2)$-overlapped}
            \label{fig:P34overlap}
        \end{figure}

        If $s$ is not $(i+1,i+2)$-overlapped in $x$, then $s$ contains an $i+1$ in $w_{i+1}(x)$ but not in $w_{i+1}(E_i(x))$. Thus, $w_{i+1}(E_i(x))$ is $w_{i+1}(x)$ with one of the $i+1$'s removed. If this removal breaks up a pair of parentheses, then $\Delta_i\eps(x,i+1) = 0$ and $\Delta_i\delta(x,i+1) = -1$ (see Figure \ref{subfig:P34nonoverlapa}). Else, $\Delta_i\eps(x,i+1) = -1$ and $\Delta_i\delta(x,i+1) = 0$ (see Figure \ref{subfig:P34nonoverlapb}). In both cases, (P3) and (P4) are satisfied.

        \begin{figure}[h!]
         \centering
            \begin{subfigure}[b]{\textwidth}
                \centering
                \includegraphics[scale = 2]{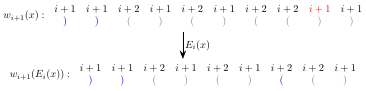}
                \caption{Removing $s$ (in red) breaks up a pair of parentheses}
                \label{subfig:P34nonoverlapa}
            \end{subfigure}
            
         \vspace{2em}
         
            \begin{subfigure}[b]{\textwidth}
                \centering
                \includegraphics[scale = 2]{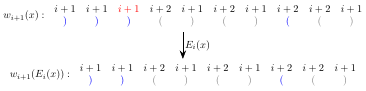}
                \caption{Removing $s$ (in red) does not break up a pair of parentheses}
                \label{subfig:P34nonoverlapb}
            \end{subfigure}

            \caption{Possible scenarios with $w_{i+1}$ when $s$ is not $(i+1,i+2)$-overlapped}
            \label{fig:P34nonoverlap}
        \end{figure}

        The case where $j = i-1$ is analogous.
    \end{proof}

    \begin{lemma}\label{lem:P5}
        The crystal operators defined in Section \ref{sec:crys} satisfy axioms (P5):
        \[ \Delta_i\delta(x,j) = 0 ~\text{implies}~ y := E_iE_jx = E_jE_ix, ~\text{and}~ \nabla_j\eps(y,i) = 0. \]
    \end{lemma}

    \begin{proof}
        If $|i-j| > 1$, then (P5) follows from the definition.

        If $j = i+1$, let $s$ be the square containing the leftmost unmatched $i+1$ in $w_i(x)$, and $t$ be the square containing the leftmost unmatched $i+2$ in $w_{i+1}(x)$. We will prove the following diagram
        \[\begin{tikzcd}[sep=small]
    	x && {E_ix} \\
    	\\
    	{E_{i+1}x} && y
    	\arrow["s", from=1-1, to=1-3]
    	\arrow["t", from=1-3, to=3-3]
    	\arrow["t"', from=1-1, to=3-1]
    	\arrow["s"', from=3-1, to=3-3]
        \end{tikzcd}\]
        i.e. $E_i$ acts on both $x$ and $E_{i+1}x$ at $s$, and $E_{i+1}$ acts on both $x$ and $E_ix$ at $t$.
        
        By definition, $E_i$ and $E_{i+1}$ act on $x$ at $s$ and $t$, respectively. From the proof of Lemma \ref{lem:P1-P2}, $\Delta_i\delta(x,j+1) = 0$ implies either $s$ is not $(i+1,i+2)$-overlapped in $x$ and removing $s$ does not break up a pair of parentheses in $w_{i+1}(x)$ (Figure \ref{subfig:P34nonoverlapb}), or $s$ is $(i+1,i+2)$-overlapped with some $(i+2)$-square $u$ in $x$ and this $i+2$ creates a new pair of parentheses in $w_{i+1}(E(x))$ (Figure \ref{subfig:P34overlapa}).

        \textit{Case 1:} $s$ is not $(i+1,i+2)$-overlapped in $x$ and removing $s$ does not break up a pair of parentheses in $w_{i+1}(x)$ (Figure \ref{subfig:P34nonoverlapb}). Observer that in this case, the sets of unmatched $(i+2)$ in $w_{i+1}(x)$ and $w_{i+1}(E_ix)$ are the same. Thus, $t$ still contains the leftmost unmatched $(i+2)$ in $w_{i+1}(E_ix)$, so $E_{i+1}$ acts on $E_ix$ at $t$.
        
        Now we prove that $E_i$ acts on $E_{i+1}x$ at $s$. The key observation is that $t$ is right of $s$ since removing $s$ does not break up a pair of parentheses in $w_{i+1}(x)$. Thus, when $E_{i+1}$ acts on $x$ turning $t$ from $i+2$ to $i+1$, this $i+1$ is still right of $s$. Hence, $s$ is still the leftmost unmatched $i+1$ in $w_i(E_{i+1}(x))$, so $E_i$ also acts on $E_{i+1}x$ at $s$. Therefore, $E_i(E_{i+1}(x)) = E_{i+1}(E_i(x))$.

        Finally, let $y := E_i(E_{i+1}(x)) = E_{i+1}(E_i(x))$, we need to prove that $\nabla_{i+1}\eps(y,i) = \eps(y,i) - \eps(F_{i+1}(y),i) = 0$. Note that $F_{i+1}(y)$ is $E_i(x)$, so this is equivalent to proving that $w_i(E_{i+1}(E_i(x)))$ and $w_i(E_i(x))$ have the same number of unmatched $i$. Recall that $E_{i+1}$ acts on $E_i(x)$ at $t$, which is right of $s$, the rightmost unmatched $i$ in $w_i(E_i(x))$. Thus, adding the $i+1$ in $t$ to $w_i(E_i(x))$ does not create a new pair of parentheses, so the number of unmatched $i$ stays the same (see Figure \ref{fig:P5case1}).
        
        \begin{figure}[h!]
            \centering
            \includegraphics[scale = 2]{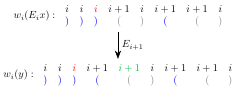}
            \caption{Adding $t$ (in green) to the right of $s$ (in red) does not change the number of unmatched $i$}
            \label{fig:P5case1}
        \end{figure}

        \textit{Case 2:} $s$ is $(i+1,i+2)$-overlapped with some $(i+2)$-square $u$ in $x$ and this $i+2$ creates a new pair of parentheses in $w_{i+1}(E_ix)$ (Figure \ref{subfig:P34overlapa}). We claim that either $t$ is right of $s$ like the above case, or $t$ is $(i,i+1)$-overlapped in $E_{i+1}x$.

        If that is not the case, then $t$ is left of $s$ and $t$ is not $(i,i+1)$-overlapped in $E_{i+1}x$. This means that $t$ is at least a row below $s$. However, the $i+2$ in $u$ creates a new pair of parentheses in $w_{i+1}(E_ix)$, so $u$ is left of $t$. Since $t$ is the leftmost $(i+2)$-square on its row by Proposition \ref{prop:well-defined}, $u$ is at least a row below $t$. However, $s$ and $u$ are exactly two rows apart. Hence, we must have the following setup: $t$ belongs to some row $p$, $s$ belongs to row $p-1$, and $u$ belongs to row $p+1$.
        
        In $x$, let
        \begin{itemize}
            \item $a$ be the number of non-$(i+1,i+2)$-overlapped $(i+2)$-squares on row $p+1$,
            \item $b$ be the number of $(i+1)$-squares on row $p$, and
            \item $c$ be the number of non-$(i,i+1)$-overlapped $i$-squares on row $p-1$.
        \end{itemize}
        Since $p-1$ and $p+1$ have some $(i+1,i+2)$-overlaps, in particular $s$ and $u$, we have
        \[ a \geq c. \]
        On the other hand, since $t$ is not $(i,i+1)$-overlapped in $E_{i+1}(x)$, all $(i+1)$-squares on row $p$ are also not $(i,i+1)$-overlapped. $s$ of row $p-1$ is the leftmost unmatched $i+1$ in $w_i(x)$, so every non-$(i,i+1)$-overlapped $(i+1)$-square of row $p$ is matched with some non-$(i,i+1)$-overlapped $i$-square of row $p-1$. This means that 
        \[ b \leq c. \]
        Similarly, $t$ is the leftmost unmatched $i+2$ in $w_{i+1}(x)$, so every non-$(i+1,i+2)$-overlapped $(i+2)$-square of $p+1$ is matched with some non-$(i+1,i+2)$-overlapped $(i+1)$-square of $p$. Furthermore, in $E_i(x)$, $u$ is also matched with another non-$(i+1,i+2)$-overlapped $(i+1)$-square of $p$. Thus, 
        \[ a+1 \leq b. \]
        Therefore, we have $a+1\leq b \leq c$, so $a < c$, contradiction. Hence, either $t$ is right of $s$ like the above case, or $t$ is $(i,i+1)$-overlapped in $E_{i+1}(x)$.

        If $t$ is right of $s$, then the argument is the same as Case 1. Else, $t$ is $(i,i+1)$-overlapped with an $i$-square $v$ in $E_{i+1}(x)$. Note that in this case, we still have $E_{i+1}$ acting on both $x$ and $E_ix$ at $t$. Also, $v$ is on row $p-2$ while $s$ is on row $p-1$, so $v$ is right of $s$. Hence, in $w_i(E_{i+1}(x))$, there is a new unmatched $i+1$ to the right of $s$, so $E_i$ still acts on $E_{i+1}x$ at $s$ (see Figure \ref{fig:P5case2}). This proves that $E_i(E_{i+1}(x)) = E_{i+1}(E_i(x))$.
        
        \begin{figure}[h!]
            \centering
            \includegraphics[scale = 2]{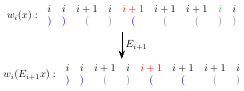}
            \caption{Removing $v$ (in green) right of $s$ (in red) creates a new unmatched $i+1$ right of $s$}
            \label{fig:P5case2}
        \end{figure}

        Finally, let $y := E_i(E_{i+1}(x)) = E_{i+1}(E_i(x))$, we prove that $E_{i+1}(E_i(x))$ and $E_i(x)$ have the same number of unmatched $i$. Recall that $E_{i+1}$ acting on $E_i(x)$ creates a new unmatched $i+1$ to the right of $s$, the rightmost unmatched $i$ in $w_i(E_i(x))$. Thus, the number of unmatched $i$ stays the same.
        
        This completes the proof for the case $j = i+1$. The case $j = i-1$ is analogous.
    \end{proof}

    \begin{lemma}\label{lem:P6}
        The crystal operators defined in Section \ref{sec:crys} satisfy axioms (P6):
        \[ \Delta_i\delta(x,j) = \Delta_j\delta(x,i) = -1 ~\text{implies}~ y:= E_iE_j^2E_ix = E_jE_i^2E_jx, \]
        \[ ~\text{and}~ \nabla_i\eps(y,j) = \nabla_j\eps(y,i) = -1 \]
    \end{lemma}

    \begin{proof}
        The condition that $\Delta_i\delta(x,j) = \Delta_j\delta(x,i) = -1$ only happens when $|i-j| = 1$. By symmetry, we only need to consider the case when $j = i+1$.
        
        Let $s$ be the square containing the leftmost unmatched $i+1$ in $w_i(x)$, and $t$ be the square containing the leftmost unmatched $i+2$ in $w_{i+1}(x)$. The condition $\Delta_i\delta(x,i+1) = \Delta_{i+1}\delta(x,i) = -1$ means that $w_{i+1}(E_ix)$ has one more unmatched $i+2$ than $w_{i+1}(x)$, let $u$ be the square containing this new unmatched $i+2$. Similarly, $w_i(E_{i+1}x)$ has one more unmatched $i+1$ than $w_i(x)$, let $v$ be the square containing this new unmatched $i+1$. Let us recall the possible scenarios involving $t$ and $v$.

        \begin{itemize}
            \item Scenario 1: $t$ is not $(i,i+1)$-overlapped in $E_{i+1}x$, and the new $i+1$ from $t$ does not create a new pair of parentheses in $w_i(E_{i+1}x)$ (see Figure \ref{subfig:P6_vandt_a}).
            \item Scenario 2: $t$ is $(i,i+1)$-overlapped with some square $t'$ in $E_{i+1}x$, and removing $t'$ breaks up a pair of parentheses in $w_i(E_{i+1}x)$ (see Figure \ref{subfig:P6_vandt_b}).
        \end{itemize}

        Note that in scenario 1, $v$ is weakly left of $t$. In particular, $v$ and $t$ may be the same square. The scenarios for $s$ and $u$ are similar.

        \begin{figure}[h!]
         \centering
            \begin{subfigure}[b]{\textwidth}
                \centering
                \includegraphics[scale = 2]{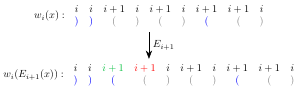}
                \caption{Adding $t$ (in red) gives new unmatched $v$ (in green) - Scenario 1}
                \label{subfig:P6_vandt_a}
            \end{subfigure}
            
         \vspace{2em}
         
            \begin{subfigure}[b]{\textwidth}
                \centering
                \includegraphics[scale = 2]{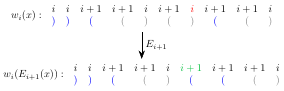}
                \caption{Removing $t'$ (in red) gives new unmatched $v$ (in green) - Scenario 2}
                \label{subfig:P6_vandt_b}
            \end{subfigure}

            \caption{Possible scenarios involving $t$ and $v$}
            \label{fig:P6_vandt}
        \end{figure}

        Since $w_i(E_{i+1}x)$ has only one more unmatched $i+1$ than $w_i(x)$, $s$ still contains one of the leftmost two unmatched $(i+1)$'s in $w_i(E_{i+1}x)$. Let $p$ be the square containing the other unmatched $i+1$. There are three different plots involving $s$, $p$, and $v$.

        \begin{itemize}
            \item Plot 1: $v$ is the leftmost unmatched $i+1$ in $w_i(E_{i+1}x)$, then $v$ is $p$, and $s$ is the second leftmost unmatched $i+1$.
            \item Plot 2: $v$ is the second leftmost unmatched $i+1$ in $w_i(E_{i+1}x)$, then $v$ is $p$ and $s$ is the leftmost unmatched $i+1$.
            \item Plot 3: $v$ is neither, then $s$ is the leftmost unmatched $i+1$, $p$ is the second leftmost unmatched $i+1$, and $v$ is right of $p$.
        \end{itemize}

        In plot 1, $E_i^2$ acts on $E_{i+1}x$ at $p$ and then $s$ (see Figure \ref{subfig:P6diagrama}). In plots 2 and 3, $E_i^2$ acts on $E_{i+1}x$ at $s$ and then $p$ (see Figure \ref{subfig:P6diagramb}). Hence, $E_i^2$ always acts on $E_{i+1}x$ at $s$ and $p$. 
        
        {\bf {Main claim}}: we shall prove that $E_i$ also acts on $E_{i+1}^2E_i(x)$ at $p$, i.e. the ? in Figure \ref{fig:P6diagram} is $p$.
        
        This would imply that $E_i^2$ acts on $E_{i+1}x$ at the same squares as $E_i$ acting on $x$ and $E_{i+1}^2E_i(x)$. By symmetry, we would have $E_{i+1}^2$ acts on $E_ix$ at the same squares as $E_{i+1}$ acting on $x$ and $E_i^2E_{i+1}(x)$. This would prove $E_iE_{i+1}^2E_ix = E_{i+1}E_i^2E_{i+1}x$.

        \begin{figure}[h!]
         \centering
            \begin{subfigure}[b]{0.4\textwidth}
                \centering
                \includegraphics[scale = 1]{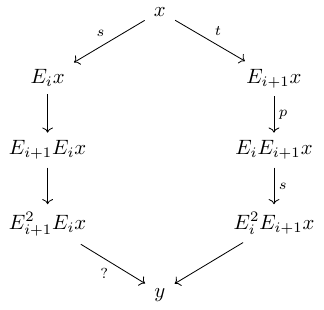}
                \caption{}
                \label{subfig:P6diagrama}
            \end{subfigure}
         \quad\quad
            \begin{subfigure}[b]{0.4\textwidth}
                \centering
                \includegraphics[scale = 1]{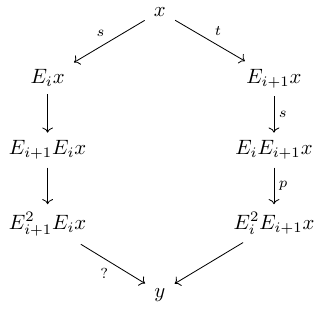}
                \caption{}
                \label{subfig:P6diagramb}
            \end{subfigure}

            \caption{}
            \label{fig:P6diagram}
        \end{figure}

        Our case analysis consists of the following cases.

        \begin{itemize}
            \item Case 1: $s$ is not $(i+1,i+2)$-overlapped in $x$.
            \begin{itemize}
                \item Case 1.1: $t$ is right of $s$.
                \begin{itemize}
                    \item Case 1.1.1: $u$ is not $(i,i+1)$-overlapped in $E_{i+1}^2E_i(x)$.
                    \item Case 1.1.2: $u$ is $(i,i+1)$-overlapped in $E_{i+1}^2E_i(x)$.
                \end{itemize}
                \item Case 1.2: $t$ is left of $u$.
            \end{itemize}
            \item Case 2: $s$ is $(i+1,i+2)$-overlapped in $x$.
            \begin{itemize}
                \item Case 2.1: The square directly above $t$ contains an $i$.
                \item Case 2.2: The square directly above $t$ does not contain an $i$.
            \end{itemize}
        \end{itemize}

        \textit{Case 1:} $s$ is not $(i+1,i+2)$-overlapped in $x$. Then, $t$ is either right of $s$ or left of $u$, this is because $u$ in this case has to contain the very first newly unmatched $i+2$ left of $s$. 

        \quad\textit{Case 1.1:} If $t$ is right of $s$, then the two leftmost unmatched $(i+2)$'s in $E_i(x)$ are $u$ and $t$. Hence, $E_{i+1}^2$ acts on $E_i(x)$ at $u$ and $t$. 
        
        \quad\quad\textit{Case 1.1.1:} $u$ is not $(i,i+1)$-overlapped in $E_{i+1}^2E_i(x)$,  then $u$ is an $i+1$ left of $s$ in $w_i(E_{i+1}^2E_i(x))$, so $u$ is matched. Thus, the only new unmatched $i+1$ in $w_i(E_{i+1}^2E_i(x))$ is in $v$. Hence, the leftmost unmatched $i+1$ in $w_i(E_{i+1}^2E_i(x))$ is still in $p$, so $E_i$ acts on $E_{i+1}^2E_i(x)$ at $p$. Figure \ref{fig:P6case1_1} gives an example of this case.
        
        \begin{figure}[h!]
            \centering
            \includegraphics[scale = 2]{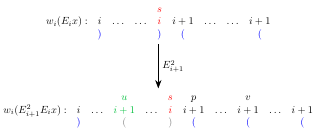}
            \caption{$u$ is matched, so the only new unmatched $i+1$ is at $v$, and $E_i$ acts on $E_{i+1}^2E_ix$ at $p$}
            \label{fig:P6case1_1}
        \end{figure}
        
        \quad\quad\textit{Case 1.1.2:} $u$ is $(i,i+1)$-overlapped in $E_{i+1}^2E_i(x)$ with some $i$-square $w$.
        
        \underline{Claim:} $w$ is left of $s$.
        
        If $w$ is right of $s$ while $u$ is left of $s$ (since we are in Case 1), then we must have the following configuration in $x$: $w$ is an $i$-square on some row $r-1$, $s$ is an $(i+1)$-square on row $r$, and $u$ is an $(i+2)$-square on row $r+1$. Furthermore, $w$ and $u$ are on the same column. Let $a$ be the number of $(i+1)$-squares right of $s$ on row $r$. These $(i+1)$-squares are all not $(i+1,i+2)$-overlapped (because $s$ is not) and hence are $(i+1)$'s in $w_{i+1}(E_ix)$. In $E_ix$, $u$ is an unmatched $i+2$, so there are at least $a$ non-$(i+1,i+2)$-overlapped $(i+2)$-squares right of $u$ on row $r+1$. Note that directly above $u$ is an $i$-square $w$, so above the other $a$ non-$(i+1,i+2)$-overlapped $(i+2)$-squares on row $r+1$ are also $i$-squares on row $r-1$. Thus, there are at least $a+1$ non-$(i,i+1)$-overlapped $i$-squares on row $r-1$. Since there are only $a$ $(i+1)$-squares right of $s$ on row $r$, it is impossible for $s$ to be unmatched. This is a contradiction. The figure below illustrates this case ($w$ is right of $s$) with $a = 3$, where $s$ is colored red, $u$ is colored blue, and $w$ is colored green.
        \[\begin{tikzcd}[sep=tiny]
    	{\textcolor{goodgreen}{i}} && i && i && i \\
    	& {\textcolor{red}{i+1}} && {i+1} && {i+1} && {i+1} \\
    	{\textcolor{blue}{i+2}} && {i+2} && {i+2} && {i+2}
        \end{tikzcd}\]

        Therefore, $w$ is left of $s$. As a result, in $w_i(E_{i+1}^2E_i(x))$, the removal of $w$ does not create a new unmatched $i+1$. Thus, similar to Case 1.1.1, the only new unmatched $i+1$ in $w_i(E_{i+1}^2E_i(x))$ is in $v$. Hence, the leftmost unmatched $i+1$ in $w_i(E_{i+1}^2E_i(x))$ is still in $p$, so $E_i$ acts on $E_{i+1}^2E_i(x)$ at $p$.

        \quad\textit{Case 1.2:} If $t$ is left of $u$, then $E_{i+1}^2$ acts on $E_i(x)$ at $t$ and some square $q$. $q$ is left of $u$, and $u$ is left of $s$, so $q$ is left of $s$. By the same reasoning as in Case 1.1, the new $i+1$ in $q$ does not create a new unmatched $i+1$ in $w_i(E_{i+1}^2E_i(x))$. Thus, the only new unmatched $i+1$ in $w_i(E_{i+1}^2E_i(x))$ is still in $v$. Hence, the leftmost unmatched $i+1$ in $w_i(E_{i+1}^2E_i(x))$ is still in $p$, so $E_i$ still acts on $E_{i+1}^2E_i(x)$ at $p$.

        \textit{Case 2:} $s$ is $(i+1,i+2)$-overlapped in $x$. In this case, we still have $u$ is left of $s$, because it either coincides with the cell $s$ is overlapped with, or is to the left of it. If $t$ is right of $s$ or left of $u$, then the argument is the same as Case 1. Thus, we only need to consider the case where $t$ is left of $s$ and right of $u$.
        
        Suppose $s$ is on row $r-1$, then $s$ is $(i+1,i+2)$-overlapped with square $s'$ on row $r+1$. By definition, $u$ is the new unmatched $i+2$ in $w_{i+1}(E_ix)$ immediately left of $s'$. Hence, in order for $t$ to be right of $u$, $t$ has to be right of $s'$. Note that $t$ is also the leftmost $(i+2)$-square of some row, so $t$ has to be the leftmost $(i+2)$-square on row $r$. Let $t'$ be the square directly above $t$ on row $r-2$. We have two subcases.

        \quad\textit{Case 2.1:} $t'$ contains an $i$. In $E_{i+1}(x)$, $t$ is $(i,i+1)$-overlapped with $t'$, so row $r-2$ has one fewer non-$(i,i+1)$-overlapped $i$-squares. Thus, row $r-1$ has one more unmatched non-$(i,i+1)$-overlapped $(i+1)$-squares. This means that $v$ and $p$ are right of $s$: $v$ because it is the new unmatched $i+1$ in row $r-1$, and $p$ because it has to be between $s$ and $v$. Then the argument follows similar to Case 1.1. The figure below illustrates this case, where $s$ is colored red, and $t$ is colored blue.
        \[\begin{tikzcd}[sep=tiny]
    	i && i \\
    	& {\textcolor{red}{i+1}} && {i+1} && {i+1} && {i+1} \\
    	{\textcolor{blue}{i+2}} \\
    	&&& {} \\
    	&&& {} \\
    	&&& {} \\
    	i && i \\
    	& {\textcolor{red}{i}} && {i+1} && {i+1} && {i+1} \\
    	{\textcolor{blue}{i+1}} &&& p && v
    	\arrow[from=4-4, to=6-4]
    	\arrow[from=9-6, to=8-6]
    	\arrow[from=9-4, to=8-4]
        \end{tikzcd}\]

        \quad\textit{Case 2.2:} $t'$ does not contains an $i$. 

        \underline{Claim:} $u$ is the leftmost $(i+2)$-square on row $r+1$.

        If $u$ is not the leftmost $(i+2)$-square on row $r+1$, then let $a$ be the number of non-$(i+1,i+2)$-overlapped $(i+2)$-squares on row $r+1$ in $x$. These are the $(i+2)$-squares left of $s'$ on row $r+1$. Because $u$ is the new unmatched $i+2$ in $w_{i+1}(E_ix)$ immediately left of $s'$, but $u$ is not on row $r+1$, $s'$ and all $a$ $(i+2)$-squares left of it on row $r+1$ are matched in $w_{i+1}(E_ix)$. This means that there are at least $a+1$ $(i+1)$-squares left of $t$ on row $r$. Note that $t'$ contains some number smaller than $i$ (or is empty), so all $(i+1)$-squares left of $t$ on row $r$ are not $(i,i+1)$-overlapped. In particular, at least $a+1$ $(i+1)$-squares left of $t$ on row $r$ are $(i+1)'s$ in $w_i(x)$. Thus, in order for $s$ to be unmatched, there are at least $a+1$ non-$(i,i+1)$-overlapped $i$-squares left of $s$ on row $r-1$ in $x$. However, below $s$ is an $(i+2)$-square $s'$, so the squares below the $a+1$ non-$(i,i+1)$-overlapped $i$-squares left of $s$ on row $r-1$ have to be $a+1$ non-$(i+1,i+2)$-overlapped $(i+2)$-squares on row $r+1$, which is a contradiction. The figure below illustrates this case with $a = 3$, where $s$ is colored red, $t$ is colored blue.

        \[\begin{tikzcd}[sep=tiny]
    	&&&&&&&&& {<i} \\
    	i && i && i && i && {\textcolor{red}{i+1}} \\
    	& {i+1} && {i+1} && {i+1} && {i+1} && {\textcolor{blue}{i+2}} \\
    	&& {i+2} && {i+2} && {i+2} && {i+2}
        \end{tikzcd}\]
        
        Thus, $u$ is the leftmost $(i+2)$-square on row $r+1$. Let $a$ be the number of non-$(i+1,i+2)$-overlapped $(i+2)$-squares on row $r+1$ in $x$. Repeating the same argument as above, the square directly above $u$ on row $r-1$ has to contain an $i$, and there are exactly $a$ non-$(i,i+1)$-overlapped $i$-squares on row $r-1$ and $a$ $(i+1)$-squares on row $r$.
        
        Since $t$ is not $(i,i+1)$-overlapped in $E_{i+1}(x)$, $p=v$ is the leftmost $(i+1)$-square on row $r$. In $E_{i+1}^2E_i(x)$, row $r-1$ gains one $i$ in $s$ and loses one $i$ that is $(i,i+1)$-overlapped with $u$, so row $r-1$ has $a$ non-$(i,i+1)$-overlapped $i$-squares. Row $r$ has one more non-$(i,i+1)$-overlapped $(i+1)$-square in $t$, so row $r$ has $a+1$ $i$-squares. Hence, the leftmost unmatched $i+1$ is in $v$ as desired. The figure below illustrates this case, where $s$ is colored red, $t$ is colored blue, and $u$ is colored green.

        \[\begin{tikzcd}[sep=tiny]
            &&&&&&& {<i} \\
    	i && i && i && {\textcolor{red}{i+1}} \\
    	& {i+1} && {i+1} && {i+1} && {\textcolor{blue}{i+2}} \\
    	{\textcolor{goodgreen}{i+2}} && {i+2} && {i+2} && {i+2} \\
    	&&& {} \\
    	&&& {} \\
    	&&& {} \\
            &&&&&&& {<i} \\
    	i && i && i && {\textcolor{red}{i}} \\
    	& {i+1} && {i+1} && {i+1} && {\textcolor{blue}{i+1}} \\
    	{\textcolor{goodgreen}{i+1}} && {i+2} && {i+2} && {i+2} \\
            & {p=v}
    	\arrow[from=5-4, to=7-4]
    	\arrow[from=12-2, to=10-2]
        \end{tikzcd}\]

        Hence, we proved that $E_i$ acts on $E_{i+1}^2E_i(x)$ at $p$. As discussed in the beginning, we have $E_iE_{i+1}^2E_i(x) = E_{i+1}E_i^2E_{i+1}(x)$.

        Finally, let $y:= E_iE_{i+1}^2E_i(x) = E_{i+1}E_i^2E_{i+1}(x)$, we need to prove that $\nabla_i\eps(y,i+1) = \nabla_{i+1}\eps(y,i) = -1$. Note that $\nabla_i\eps(y,i+1) = -1$ is equivalent to $w_{i+1}(E_{i+1}^2E_i(x))$ having one more unmatched $i+1$ than $w_{i+1}(y)$. Recall that $y$ is obtained from $E_{i+1}^2E_i(x)$ by changing the $i+1$ in $p$ to an $i$.

        In $w_{i+1}(E_iE_{i+1}^2(x))$, $t$ contains an unmatched $i+1$. Thus, the statement is immediate if $p$ is left of $t$. This is because changing the $i+1$ in $p$ to an $i$ in this case either remove an $i+1$ left of $t$ or add an $i+2$ left of $t$ in $w_{i+1}(E_iE_{i+1}^2(x))$.
        
        Now we consider the case where $p$ is right of $t$. Note that $v$ is left of $t$ implies $p$ is left of $t$, so this only happens when $v$ is right of $t$. For $v$ to be right of $t$, we must have the following configuration in $x$: $t$ is an $(i+2)$-square on row $r+1$, the square $t'$ directly above $t$ on row $r-1$ contains an $i$, and $v$ and $p$ are $(i+1)$-squares on row $r$, with $p$ being weakly left of $v$. Let $a$ be the number of non-$(i,i+1)$-overlapped $(i+1)$-squares right of $v$ on row $r$. For $v$ to be unmatched in $w_{i}(E_{i+1}(x))$, there must be exactly $a+1$ non-$(i,i+1)$-overlapped $i$-squares on row $r-1$. The leftmost such square is $t'$ since $t$ is the leftmost $(i+2)$-square on row $r+1$. Thus, row $r+1$ has at most $a+1$ non-$(i+1,i+2)$-overlapped $(i+2)$-squares. However, $t$ is unmatched in $w_{i+1}(x)$, so $v$ has to be $(i+1,i+2)$-overlapped. This implies that $p$ is also $(i+1,i+2)$-overlapped with some $(i+2)$-square on row $r+2$ because $p$ is weakly left of $v$ on row $r$.

        Therefore, when $E_i$ acts on $E_{i+1}^2E_i(x)$, changing the $i+1$ in $p$ into an $i$, $w_{i+1}(y)$ has an extra $i+2$ on row $r+2$. This $(i+2)$-square is left of $t$, which contains an unmatched $i+1$. Hence, $w_{i+1}(y)$ has one fewer unmatched $i+1$ than $w_{i+1}(E_{i+1}^2E_i(x))$. This completes the proof.
    \end{proof}

    \begin{lemma}\label{lem:P5'6'}
        The crystal operators defined in Section \ref{sec:crys} satisfy axioms (P5') and (P6'):
        \begin{itemize}
            \item[(P5')] $\nabla_i\eps(x,j) = 0$ implies $y := F_iF_jx = F_jF_ix$, and $\Delta_j\delta(y,i) = 0$.
            \item[(P6')] $\nabla_i\eps(x,j) = \nabla_j\eps(x,i) = -1$ implies $y:= F_iF_j^2F_ix = F_jF_i^2F_jx$, and $\Delta_i\delta(y,j) = \Delta_j\delta(y,i) = -1$
        \end{itemize}
    \end{lemma}

    \begin{proof}
        The proof for this lemma is dual to those of Lemma \ref{lem:P5} and \ref{lem:P6}.
    \end{proof}

\section{Generalized Littlewood-Richardson rule}\label{sec:LR-rule}

    We begin with the analogue of Yamanouchi tableaux.

    \begin{definition}[Yamanouchi shuffle tableau]\label{def:Yamanouchi}
        A shuffle tableau $T$ of shape $\mu\oslash \nu$ is called \textbf{Yamanouchi} if no crystal operator $E_i$ can be applied on $T$.
    \end{definition}

    As an immediate corollary of Theorem \ref{thm:typeA-crystal} we obtain the following generalized Littlewood-Richardson rule.

    \begin{thm}\label{thm:LR-rule}
        For any partitions $\mu,\nu$, any Temperley-Lieb type $\tau$, and any partition $\lambda$, the coefficient of the Schur function $s_\lambda$ in $\Imm^{\TL}_\tau(A_{\mu,\nu})$ is the number of Yamanouchi shuffle tableaux of shape $\mu\oslash\nu$, Temperley-Lieb type $\tau$, and content $\lambda$.
    \end{thm}

    \begin{example}\label{exp:LR-rule}
        Let $\mu = (5,5,4,3)$, $\nu = (3,2,2,0)$, and $\tau$ be the following Temperley-Lieb type
        \[ \resizebox{!}{0.15\textwidth}{
            \begin{tikzpicture}
                \draw (0,0) node[anchor=center] {$L_1$};
                \draw (0,-1) node[anchor=center] {$L_2$};
                \draw (0,-2) node[anchor=center] {$L_3$};
                \draw (0,-3) node[anchor=center] {$L_4$};
                \draw (3,0) node[anchor=center] {$R_1$};
                \draw (3,-1) node[anchor=center] {$R_2$};
                \draw (3,-2) node[anchor=center] {$R_3$};
                \draw (3,-3) node[anchor=center] {$R_4$};
                \filldraw[black] (0.5,0) circle (2pt);
                \filldraw[black] (0.5,-1) circle (2pt);
                \filldraw[black] (0.5,-2) circle (2pt);
                \filldraw[black] (0.5,-3) circle (2pt);
                \filldraw[black] (2.5,0) circle (2pt);
                \filldraw[black] (2.5,-1) circle (2pt);
                \filldraw[black] (2.5,-2) circle (2pt);
                \filldraw[black] (2.5,-3) circle (2pt);
                \draw[thick] [bend right = 90, looseness=1.25] (0.5, -1) to (0.5, 0);
                \draw[thick] [bend right = 90, looseness=1.25] (0.5, -3) to (0.5, -2);
                \draw[thick] [bend right = 90, looseness=1.25] (2.5, -1) to (2.5, -2);
                \draw[thick] [bend right = 90, looseness=1.25] (2.5, 0) to (2.5, -3);
            \end{tikzpicture}
        }. \]
        The coefficient of $s_{(5,4,1)}$ in $\Imm^{\TL}_\tau(A_{\mu,\nu})$ is 2, and we have the following two Yamanouchi shuffle tableaux of shape $\mu\oslash\nu$
        \[ \begin{ytableau}   
        \none & \none & \none & 1 & \none & 1 \\ 
        \none & \none & 1 & \none & 1 & \none & 1 \\ 
        \none & \none & \none & 2 & \none & 2 \\ 
        2 & \none & 2 & \none & 3
        \end{ytableau}, \quad\quad \begin{ytableau}    
        \none & \none & \none & 1 & \none & 1 \\ 
        \none & \none & 1 & \none & 1 & \none & 2 \\ 
        \none & \none & \none & 2 & \none & 3 \\ 
        1 & \none & 2 & \none & 2
        \end{ytableau}. \]
        We encourage the readers to check that the two tableaux above indeed have Temperley-Lieb type $\tau$.
    \end{example}

    \begin{remark} \label{remark:usualLR}
        Theorem \ref{thm:LR-rule} is a generalization of the classical Littlewood-Richardson rule. Let $\tau_0$ be the determinant, for a shuffle tableau $T$ to have Temperley-Lieb type $\tau_0$, there has to be no pair
        \[ \begin{tikzcd}[sep = tiny]
    	{} & j \\
    	i & {}
    	\arrow[no head, from=2-1, to=1-1]
    	\arrow[no head, from=1-2, to=2-2]
        \end{tikzcd}. \]
        This means that the entries in $T$ are strictly increasing along the anti-diagonals. Hence, rows and anti-diagonals of $T$ form a semi-standard Young tableau. For instance, we have
        \[ \begin{ytableau}   
        \none & \none & \none & \none & \none & \none & \none & 1 & \none & 1 \\ 
        \none & \none & \none & \none & 1 & \none & 2 & \none & 2 \\ 
        \none & \none & \none & 2 & \none & 3 \\ 
        1 & \none & 4
        \end{ytableau}~~~\longrightarrow~~~ \begin{ytableau}   
        \none & \none & 1 & 1 \\ 
        \none & 1 & 2 & 2 \\ 
        \none & 2 & 3 \\ 
        1 & 4
        \end{ytableau}. \]
        This gives a bijection between shuffle tableaux of shape $\mu\oslash\nu$, type $\tau_0$ and semi-standard Young tableaux of shape $\mu'/\nu' := (\mu-\rho)/(\nu-\rho)$, where $\rho = (n-1,\ldots,1,0)$. Furthermore, since the entries are strictly increasing along the anti-diagonals of $T$, so the difference between consecutive entries on the same column in $T$ is at least 2. Thus, each reading word $w_i(T)$ is obtained by reading $T$ row by row, so the crystal operators on $T$ are the same as on the corresponding semi-standard Young tableau. Thus, so the condition that no crystal operator $E_i$ can be applied to $T$ is equivalent to the condition of classical Yamanouchi tableaux. Therefore, the Yamanouchi shuffle tableaux of type $\tau_0$ biject with the classical Yamanouchi tableaux.
    \end{remark}

\bibliography{bibliography}
\bibliographystyle{alpha}

\end{document}